\documentclass[11pt,a4paper,reqno]{amsart} 
\usepackage{anysize}
\marginsize{3.5cm}{3.5cm}{2.2cm}{2.2cm}

\usepackage{hyperref} 
\hypersetup{
    colorlinks=true,       
    linkcolor=red,          
    citecolor=blue,        
    filecolor=magenta,      
    urlcolor=cyan           
}
\usepackage{standalone}
\usepackage[english]{babel}
\usepackage{amsmath}
\usepackage{mathtools}
\usepackage{amsfonts,amssymb}
\usepackage{enumerate}
\usepackage{mathrsfs}
\usepackage{amsthm}
\usepackage{tikz}
\usepackage{tikz}
\usetikzlibrary{intersections,decorations.markings}
\usetikzlibrary{calc}

\newfont{\bb}{msbm10 at 12pt}
\newfont{\bbt}{msbm10 at 9pt}
\def\r{\mathbb R}
\def\d{\mathbb D}
\def\n{\mathbb N}
\def\z{\mathbb Z}
\def\t{\mathbb T}

\def\b{\hbox{B}}

\def\fr{\mathscr F}
\def\sr{\mathscr S}
\def\dr{\mathscr D}
\def\w{\mathscr W}
\def\er{\mathscr E}

\DeclareMathOperator{\IM}{Im}
\DeclareMathOperator{\RE}{Re}
\DeclareMathOperator{\supp}{supp}
\DeclareMathOperator{\D}{\langle D \rangle}

\DeclareMathOperator{\op}{Op}

\newcommand{\norm}[1]{\left\Vert #1 \right\Vert}
\newcommand{\abs}[1]{\left\vert #1 \right\vert}
\newcommand{\set}[1]{\left\{#1\right\}}
\newcommand{\eps}{\epsilon}

\newcommand{\prt}[1]{\left( #1 \right)}
\newcommand{\crch}[1]{\left[ #1 \right]}

\usepackage{hyperref}
\hypersetup{
    colorlinks=true,
    linkcolor=red,
    filecolor=red,      
    urlcolor=blue,
    citecolor=red,
    pdftitle={On the well-posedness of the weakly dispersive Burgers type equations},
    bookmarks=true,
    pdfpagemode=FullScreen,
}

\numberwithin{equation} {section}

\theoremstyle{plain}\newtheorem{lemma}{Lemma}[section]
\theoremstyle{plain}\newtheorem{proposition}{Proposition}[section]
\theoremstyle{plain}\newtheorem{theorem}{Theorem}[section]
\theoremstyle{plain}\newtheorem{definition}{Definition}[section]
\theoremstyle{plain}\newtheorem{remark}{Remark}[section]
\theoremstyle{plain}\newtheorem{corollary}{Corollary}[section]
\theoremstyle{plain}\newtheorem{conjecture}{Conjecture}[section]
\theoremstyle{plain}\newtheorem{notation}{Notation}[section]
\theoremstyle{plain}\newtheorem{definition-notation}{Definition-Notation}[section]
\theoremstyle{plain}\newtheorem{definition-proposition}{Definition-Proposition}[section]
\numberwithin{equation}{section}

\title{On the Cauchy problem for dispersive Burgers type equations}
\author{Ayman Rimah Said}

\def\notina[#1]#2{\begingroup\def\thefootnote{\fnsymbol{footnote}}\footnote[#1]{#2}\endgroup}

   \def\XXint#1#2#3{{\setbox0=\hbox{$#1{#2#3}{\int}$}
        \vcenter{\hbox{$#2#3$}}\kern-.5\wd0}}

\begin{document}

\notina[0]{ PhD student at l'IMO, Paris Sud University and Centre Borelli, ENS Paris Saclay. email: \url{aymanrimah@gmail.com}.}

\begin{abstract}
We study the paralinearised weakly dispersive Burgers type equation: 
$$\partial_t u+T_u \partial_xu+\partial_x \abs{D}^{\alpha-1}u=0,\ \alpha \in ]1,2[,$$
which contains the main non linear ``worst interaction" terms, that is low-high interaction terms, of the usual weakly dispersive Burgers type equation:
$$\partial_t u+u\partial_x u+\partial_x \abs{D}^{\alpha-1}u=0,\ \alpha \in ]1,2[,$$ with $u_0 \in H^s(\d)$, where $\d=\t \text{ or } \r$. 

Through a paradifferential complex Cole-Hopf type gauge transform we introduced in \cite{Ayman20}, we prove a new a priori estimate in $H^s(\d)$ under the control of $\norm{D^{2-\alpha}\left(u^2\right)}_{L^1_tL^{\infty}_x}$, improving upon the usual hyperbolic control  $\norm{\partial_x u}_{L^1_tL^\infty_x}$. Thus we eliminate the ``standard" wave breaking scenario in case of blow up as conjectured in \cite{Saut14}.

For $\alpha\in ]2,3[$ we show that we can completely conjugate the paralinearised dispersive Burgers equation to a semi-linear equation of the form:
$$\partial_t \left[T_{e^{iT_{p(u)}}}u\right]+ \partial_x \abs{D}^{\alpha-1}\left[T_{e^{iT_{p(u)}}}u\right]=T_{R(u)}u,\ \alpha \in ]2,3[,$$
where $T_{p(u)}$ and $T_{R(u)}$ are paradifferential operators of order $0$ defined for $u\in L^\infty_t C^{(2-\alpha)^+}_*$. 
\end{abstract}

\maketitle

\vspace{-5mm}

\tableofcontents

\vspace{-7mm}

\section{Introduction}
This paper is concerned with the well-posedness of the paralinearised ``weak" dispersive perturbations of the Burgers equation:
\begin{equation}\label{glob CP DBeq_intro_BDP para'}
\partial_t u+T_u \partial_x u+\partial_x \abs{D}^{\alpha-1}u=0,\ u_0\in H^s,
\end{equation}
which is derived from the ``weak" dispersive perturbations of the Burgers equation:
\begin{equation}\label{glob CP DBeq_intro_BDP}
\partial_t u+u\partial_x u+\partial_x \abs{D}^{\alpha-1}u=0,\text{ where } \alpha \in ]1,2] \text{ and } \abs{D}=\op(\abs{\xi}).
\end{equation}

For $\alpha=2$ in \eqref{glob CP DBeq_intro_BDP}, we have the usual Benjamin-Ono equation, for $\alpha=3$ we have the KdV equation, for $\alpha=\frac{1}{2}$ and $\alpha=\frac{3}{2}$ respectively this equation is a ``toy" model for the system obtained by paralinearisation and symetrisation of water waves system with and without surface tension \cite{Alazard15,Alazard18,Alazard16,Alazard11,Alazard13}.

The Cauchy problem associated to \eqref{glob CP DBeq_intro_BDP} has been extensively studied in the literature for a comprehensive and complete overview of those equations and their link to other problems coming from mechanical fluids and dispersive non linear equations in physics we refer to J-C. Saut's \cite{Saut13,Saut18}.

For $\alpha\geq 2$ the Cauchy problem is now  very well understood where essentially four main techniques come into play to understand it. A first approach is to use smoothing effects and refined Strichartz estimates see \cite{Ponce91,Koch03}. A second approach is to use time dependent frequency localised spaces as in \cite{Ionescu07}. A third, first introduced by Tao in \cite{Tao04} is to use a gauge transform to eliminate the worst interaction terms, this can be combined with frequency localised spaces as was done \cite{Burq08,MolinetPilod12} for the Benjamin-Ono equation and \cite{Herr10} for $\alpha\in ]2,3[$, which to the author's knowledge is the only time the gauge transform was used to improve upon the local well-posedness of dispersive Burgers equation in the fractional dispersion case. A last approach is to use a gauge transform combined with a normal form transform as was done in \cite{Ifrim17} to give a simple and elegant proof to the $L^2$ well-posedness of the Benjamin-Ono equation.

Except for the first approach, in the words of \cite{Schippa19}, those techniques face major technical difficulties for $\alpha<2$ and seem to completely fail. The goal of this paper is to show that using the gauge transform introduced in \cite{Ayman20}, the last approach can be still carried for $1\leq\alpha<2$ and it gives $H^s(\d)$ estimates under the control of $\norm{D^{2-\alpha}\left(u^2\right)}_{L^1_tL^{\infty}_x}$, improving upon the known hyperbolic control  $\norm{\partial_x u}_{L^1_tL^\infty_x}.$ To the author's best knowledge this is the first time a gauge transform technique was carried out to improve upon the local well-posedness of the weakly dispersive Burgers equation.

We will also show that for $2\leq\alpha\leq 3$, this gauge transform can be efficiently used to completely conjugate the paralinearised dispersive Burgers equation to a semi-linear dispersive equation under the control of $\norm{u}_{L^\infty_t C^{(2-\alpha)^+}_{*}}$. Again to the author's best knowledge this is the first time such a transformation is carried out outside the integrable cases, that is $\alpha=2$ and $\alpha=3$. For those cases, that is the Benjamin-Ono and the KdV equations, suitable Birkhoff coordinates were constructed to ``diagonalise" the infinite dimension Hamiltonian, for this we refer to the pioneering works of G\'erard, Kappeler and Topalov \cite{Gerard20,Gerard20-1,Kappler06}.

First we start with the simplest Cauchy theory for this equation that only uses hyperbolic estimates reads:
\begin{theorem}\label{glob CP DBeq_intro_th simple CP}
Consider three real numbers $\alpha\in [0,2[$, $s\in]1+\frac{1}{2},+\infty[$, $r>0$ and $u_0 \in H^{s}(\d)$. Then there exists $C_s>0$ such that for $0<T< \frac{C_s}{r+\norm{\partial_x u_0}_{L^\infty(\d)}}$ and all $v_0$ in the ball $\b(u_0,r)\subset  H^{s}(\d)$ there exists a unique $v\in C([0,T],H^s(\d))$ solving the Cauchy problem:
\begin{equation}\label{glob CP DBeq_intro_th simple CP_mod on r}
\begin{cases} 
\partial_t v+v\partial_xv+\abs{D}^{\alpha-1} \partial_xv=0 \\
v(0,\cdot)=v_0(\cdot)
\end{cases}.
\end{equation}
Moreover, for all $\delta >0$ and all of $\mu \in [0,s],$ we have for all $t\in [0,T]$:
\begin{equation}
  \norm{v(t)}_{H^{\mu}(\d)} \leq e^{C_\mu \norm{\partial_x v}_{L^1([0,T],L^\infty(\d))}}  \norm{v_0}_{H^{\mu}(\d)}. 
\end{equation}
Taking $v_0 \in \b(u_0,r)$, and assuming moreover that $u_0 \in H^{s+1}(\d)$ then for all $t\in [0,T]$:
\begin{align}
\norm{(u-v)(t)}_{H^{s}(\d)} &\leq e^{C_s( \norm{\partial_x (u,v)}_{L^1([0,t],L^\infty(\d))}+ t\norm{u_0}_{H^{s+1}(\d)})}  \norm{u_0-v_0}_{H^{s}(\d)}.
\end{align}
\end{theorem}

Equation \eqref{glob CP DBeq_intro_th simple CP_mod on r} is known (see \cite{Saut13}) to be quasi-linear for $\alpha<3 $ that is the flow map is not regular $(C^1)$ and as such it cannot be solved through a fixed point scheme. This lack of regularity comes from the ``bad" low-high frequency interaction $u_{\text{low}}\partial_x u_{\text{high}}$.

Despite of this lack of regularity of the flow map, Tao used in \cite{Tao04} a generalised Cole-Hopf complex transformation to prove global well posedness of the Benjamin-Ono equation in $H^1(\r)$. This technique was extensively used to push down the well posedness threshold to $L^2(\r)$ for the Benjamin-Ono equation in \cite{Ionescu07}. This contrasts the fact that for the Burgers equation, that is $\alpha=1$, the Cauchy problem is ill-posed in $H^s,s\leq \frac{3}{2}$ as shown in \cite{Linares14}.

Thus the general idea is to understand the ``interaction" between the nonlinearity and dispersion. The following quantities are conserved by the flow associated to \eqref{glob CP DBeq_intro_BDP}:
\begin{equation}\label{glob CP DBeq_intro_mass conserv}
\norm{u(t)}_{L^2}=\norm{u_0}_{L^2}, \text{ and, }
\end{equation}

\begin{equation}\label{glob CP DBeq_intro_energy conserv}
H(u)=\int_{\r} \abs{D^{\frac{\alpha-1}{2}}u}^2(t,x)dx+\frac{1}{3}\int_{\r}u^3(t,x)dx=H(u_0).
\end{equation}

By the the Sobolev embedding $H^{\frac{1}{6}}\hookrightarrow L^3$, $H(u)$ is well defined for $\alpha\geq 1+\frac{1}{3}$. Moreover \eqref{glob CP DBeq_intro_BDP} is invariant under the scaling transformation:
\[
u_{\lambda}=\lambda^{\alpha-1}u(\lambda^{\alpha}t,x),
\]
for any positive $\lambda$. We have $\norm{u_\lambda(t,\cdot)}_{\dot{H}^s}=\lambda^{\alpha+s-\frac{3}{2}}\norm{u(\lambda^\alpha t,\cdot)}_{\dot{H}^s}$, thus the critical index corresponding to \eqref{glob CP DBeq_intro_BDP} is $s_c=\frac{3}{2}-\alpha$. In particular, \eqref{glob CP DBeq_intro_BDP} is $L^2$ critical for $\alpha=\frac{3}{2}$.

In the ``low" dispersion case, that is $\alpha \leq 2$ a complete numerical study was carried out by Klein and Saut in \cite{Saut14} and conjectured among other things the following.
\begin{conjecture}\label{glob CP DBeq_intro_conjecture Saut}
\begin{enumerate}
\item For $\alpha \leq 1$ solutions blow up in finite time and do so through a wave breaking scenario, that is $\displaystyle \lim_{t \rightarrow T^*}\norm{u(t)}_{L^\infty}$ stays bounded while $\displaystyle \norm{\partial_x u(t)}_{L^\infty}\rightarrow +\infty$.
\item For $\alpha > 1$ we have global in time existence for small initial data.

\item For $1<\alpha \leq \frac{3}{2}$ large solutions blow up in finite time and do so through a ``dispersive" blow up scenario, that is $\displaystyle \norm{ u(t)}_{L^\infty_x}\rightarrow +\infty$.

\item For $\alpha > \frac{3}{2}$ solutions exist globally in time.
\end{enumerate}
\end{conjecture}
In \cite{Castro10} and \cite{Hur12} blow up is proven for $\alpha< 1$ and in \cite{Hur17,Hur18} it is shown that for $\alpha <\frac{2}{3}$ the only possible blow up scenario is a wave breaking one, a simpler proof can be found in \cite{Saut20'}. 

In \cite{Linares14} using Strichartz estimate well posedness is proved for $s>\frac{3}{2}-\frac{3(\alpha-1)}{8}$ and $\alpha>1$, proving that even for very low dispersion the threshold of well posedness can be improved which again contrasts with the Burgers equation ($\alpha=1$) where the equation is shown is shown  to be ill-posed for $s=\frac{3}{2}$ in \cite{Linares14}.

This was improved upon in \cite{Molinet18} using an adapted version of the ``I-method" and refined Strichartz estimates in co-normal Bourgain type spaces. They proved well posedness for $s>\frac{3}{2}-\frac{5(\alpha-1)}{4}$ and $\alpha>1$, thus proving by the conservation of $H(u)$ and scaling, global well posedness for $\alpha>1+\frac{6}{7}$. This is the first and to the authors knowledge only result proving global existing results for $\alpha<2$.

Finally for $\alpha \geq 2$ the Cauchy problem is much better understood. For the Benjamin-Ono equation on $\r$, to the authors knowledge the best known result is $L^2$  global well-posedness derived in \cite{Ionescu07}. Recently Patrick G\'erard, Thomas Kappeler and Peter Topalov proved in \cite{Gerard20} global well-posedness for the periodic Benjamin-Ono equation all the way down to $s>-\frac{1}{2}$ and ill-posedness for $s< -\frac{1}{2}=s_c$, the critical Sobolev exponent. For $\alpha\in ]2,3[$ on the real line, the best known local well posedness result is for $s\geq \frac{3}{4}(2-\alpha)$ under a low frequency condition given in \cite{Herr07} and in $L^2$ without the low frequency condition in \cite{Herr10}. For the KdV equation, for both the periodic and real line cases the Cauchy problem is globally well posed on $H^{-1}(\r)$ as shown in \cite{Kappler06,Killip19}, which is the best possible well posedness result, that is the KdV equation is ill-posed for $s<-1$ as shown in \cite{Molinet12}.

The remarkable well-posedness results for $\alpha=\set{2,3}$ uses the integrability of the Benjamin-Ono equation and the KdV equation and the construction of  Birkhoff coordinates and thus cannot be extended to the case $\alpha\neq \set{2,3}$.

\begin{remark}
It's interesting to compare \eqref{glob CP DBeq_intro_BDP} to the ``fractal" Burgers equation, that is the Burgers equation with a dissipative term:
\begin{equation}\label{glob CP DBeq_intro_Bdis eq}
\partial_t u+u\partial_x u+(-\Delta)^{\frac{\alpha}{2}}u=0, \ \alpha \geq 0.
\end{equation}
For $\alpha =2$, \eqref{glob CP DBeq_intro_Bdis eq} is the usual Hopf equation. The local and global Cauchy problem associated to \eqref{glob CP DBeq_intro_Bdis eq} is very well understood, we refer to \cite{Kieslev08} for a complete solution to the problem. For $\alpha<1$, large solution of \eqref{glob CP DBeq_intro_Bdis eq} blow up in finite time and that through a wave breaking mechanism. For $\alpha \geq 1$ solutions exist globally in time. 

This contrasts with \eqref{glob CP DBeq_intro_BDP} in several directions:
\begin{itemize}
\item the change of local/global well posedness is conjectured to happen at $\alpha= \frac{3}{2}$ and not $1$,
\item the conjectured existence of a new nonlinear blow up regime for $\alpha \in ]1,\frac{3}{2}]$,
\item the drastic difference in the behavior of the global Cauchy problem for \eqref{glob CP DBeq_intro_BDP} and  \eqref{glob CP DBeq_intro_Bdis eq} for $\alpha=1$.
\end{itemize} 
\end{remark}

In this paper we will look more closely to the equation:
\begin{equation*}
\partial_t u+T_u\partial_x u+\partial_x \abs{D}^{\alpha-1}u=0,\ u_0\in H^s.
\end{equation*}
The term $T_u \cdot $ is the called the paraproduct with $u$. A rigorous review of paraproducts, paradifferential operators and paradifferential calculus is given in Appendix \ref{paracomposition_Notions of microlocal analysis_Paradifferential Calculus}. As the 2 main theorems in this paper use extensively this language we give an intuitive interpretation of those concepts in the following paragraph so that the reader unfamiliar with this language can get a good grasp of the statements without having to go through Appendix \ref{paracomposition_Notions of microlocal analysis_Paradifferential Calculus} first.
 
\subsubsection*{\textbf{$\bullet$ Paraproducts and Paradifferential operators}} 
For the sake of this discussion let us pretend that $\partial_x$ is left-invertible with a choice of $\partial_x^{-1}$ that acts continuously from $H^s$ to $H^{s+1}$. We follow here analogous ideas to the ones presented by Shnirelman in \cite{Shnirelman05}. One way to define the paraproduct of two functions $f,g\in H^s$ with $s$ sufficiently large is: we differentiate $fg$ $k$ times, using the Leibniz formula, and then restore the function $fg$ by the $k$-th power of $\partial_x^{-1}$:
 \begin{align*}
 fg&=\partial_x^{-k}\partial_x^{k}(fg)\\
 	&=\partial_x^{-k}\big(g\partial_x^k f+k\partial_x g\partial_x^{k-1} f+\dots+k\partial_x f\partial_x^{k-1} g+g\partial_x^k f  \big)\\
 	&=T_g f+T_fg+R,
 \end{align*}
 where,
 \[T_gf=\partial_x^{-k}\big(g\partial_x^k f\big), \ \ T_fg=\partial_x^{-k}\big(f\partial_x^k g\big),\]
 and $R$ is the sum of all remaining terms. The key observation is that if $s>\frac{1}{2}+k$, then $g \mapsto T_fg$ is a continuous operator in $H^s$ for $f  \in H^{s-k}$. The remainder $R$ is a continuous bilinear operator from $H^s$ to $H^{s+1}$. The operator $T_fg$ is called the paraproduct of $g$ and $f$ and can be interpreted as follows. The term $T_fg$ takes into play high frequencies of $g$ compared to those of $f$ and demands more regularity in $g\in H^s$ than $f \in H^{s-k}$ thus the term $T_fg$ bears the "singularities" brought on by $g$ in the product $fg$. Symmetrically $T_gf$ bears the "singularities" brought on by $f$ in the product $fg$ and the remainder $R$ is a smoother function ($H^{s+1}$) and does not contribute to the main singularities of the product. Notice that this definition uses a "general" heuristic from PDE that is the worst terms are the highest order terms (ones involving the highest order of differentiation). Now to make such a definition rigorous, we quantify this frequency comparison. The starting point is the product formula:
 \[
 fg(x)=\frac{1}{(2\pi)^2}\int_{\r}\int_{\r}e^{ix\cdot(\xi_1+\xi_2)}\fr(f)(\xi_1)\fr(g)(\xi_2)d\xi_1d\xi_2.
 \]
 Now if for some parameters $B>1,b>0$ one defines a cut-off function:
 \[
\psi^{B,b}(\eta,\xi)=0 \text{ when }
\abs{\xi}< B\abs{\eta}+b,
\text{ and }
\psi^{B,b}(\eta,\xi)=1 \text{ when } \abs{\xi}>B\abs{\eta}+b+1,
\]
then one can rigorously define the paraproduct as
\[
T^{B,b}_{g}f(x)=T_{g}f(x)=\frac{1}{(2\pi)^2}\int_{\r}\int_{\r}\psi^{B,b}(\xi_2,\xi_1)e^{ix\cdot(\xi_1+\xi_2)}\fr(f)(\xi_1)\fr(g)(\xi_2)d\xi_1d\xi_2.
\]

 To get a good intuition of a paradifferential operator $T_p$ with symbol $p\in \Gamma^\beta_\rho$, as a first gross approximation, one can think of $T_p$ as the composition of a paraproduct $T_f$ with Fourier multiplier $m(D)$, that is:
 \[
 T_p\approx T_f m(D), \text{ with } f\in W^{\rho,\infty} \text{ and }m \text{ is of order }\beta.
 \]
 Indeed following Coifman and Meyer's symbol reduction Proposition $5$ of \cite{Coifman78}, one can show that linear combinations of composition of a paraproduct with a Fourier multiplier are dense in the space of paradifferential operators.\\

Getting back to the main problem:
\begin{equation}\label{glob CP DBeq_intro_BDP para}
\partial_t u+T_u\partial_x u+\partial_x \abs{D}^{\alpha-1}u=0,\ u_0\in H^s.
\end{equation}
The modification made to pass from \eqref{glob CP DBeq_intro_BDP} to \eqref{glob CP DBeq_intro_BDP para} is that we dropped the the remainder terms:
 \[
\frac{\partial_x R(u,u)}{2} \text{ and } T_{\partial_x u}u.
\]
The motivations to study the paralinearised version of the equations are the following:
\begin{enumerate}
\item Equation \eqref{glob CP DBeq_intro_BDP para} still contains the main ``bad" term $u_{low}\partial_x u_{high}$. As remarked in \cite{Tao04} and \cite{Burq08} this is the main term obstructing straightforward estimates in $X^{0^+,\frac{1}{2}^+}$ for $\alpha=2$.

\item Indeed looking through the literature \cite{Burq08} and \cite{Ionescu07}, the neglected terms can be treated in localised Besov-Bourgain type spaces in our threshold of regularity. By contrast the results on \eqref{glob CP DBeq_intro_BDP para} are simpler to write because they can be completely described in the usual Sobolev spaces.
\end{enumerate}
Thus to keep our presentation clear and put the key ideas forward in treating the ``worst" terms, we opted to present in this paper the results on the paralinearised version of the equation \eqref{glob CP DBeq_intro_BDP para} and give the results on the ``full" equation \eqref{glob CP DBeq_intro_BDP} in Bourgain type spaces in a forthcoming work where the action of the gauge transform used here will be studied in such spaces.

Now we give the first theorem of this paper.
\begin{theorem}\label{glob CP DBeq_intro_Para th CP imprv enrgy est}
Consider two real numbers $ \alpha\in ]1,2[$, $s\in ]1+\frac{1}{2},+\infty[$. Then for all $v_0\in H^s(\d)$ and all $r>0$ there exists $C_s>0$ such that for $0<T< \frac{C_s}{r+\norm{\partial_x u_0}_{L^\infty(\d)}}$ and all $u_0$ in the ball $\b(v_0,r)\subset  H^{s}(\d)$ there exists a unique $u\in C([0,T],H^s(\d))$ solving the Cauchy problem:
\begin{equation}\label{glob CP DBeq_intro_Para th CP imprv enrgy est_ eq BDP para avec cutoff}
\partial_t u+T^{B,b}_u\partial_x u +\partial_x \abs{D}^{\alpha-1}u=0,\ u_0\in H^s.
\end{equation}

Moreover there exist $B>1$ such that for $\norm{u_0}_{H^{\prt{\frac{3}{2}-\alpha}^+}}$ sufficiently small we have the estimate:
 \begin{equation}\label{glob CP DBeq_intro_Para th CP imprv enrgy est_energy est}
\norm{u(t)}_{H^s}\leq e^{C\norm{D^{2-\alpha}\left(u^2\right)}_{L^1_tL^{\infty}_x}}\norm{u_0}_{H^s}, \text{ for } 1<\alpha<2.
\end{equation}
\end{theorem}
\vspace{0,5cm}
Some comments on the previous theorem are in order:
\begin{itemize}
\item the apriori estimate \eqref{glob CP DBeq_intro_Para th CP imprv enrgy est_energy est} is not enough to improve upon the local well-posedness theory, indeed we need an extra estimate on the difference of two solutions. A straightforward computation shows that taking the difference of two solutions $u-v$ we get:
\begin{multline*}
\partial_t(u-v)+T_{i\xi \abs{\xi}^{\alpha-1}}(u-v)+\underbrace{\partial_x [T_u (u-v)]-T_{\frac{\partial_x u}{2}}(u-v)}_{(1)}\\+\underbrace{\partial_x [T_{u-v} v]-T_{\frac{\partial_x {u-v}}{2}}v}_{(2)}=0.
\end{multline*}
Term $(1)$ can be treated using the gauge transform but term $(2)$ is not a paradifferential operator in the variable $u-v$ and can not be treated in our current restricted paradifferential-Sobolev space setting. Indeed term $(2)$ has the same structure as the residual terms we dropped to get equation \eqref{glob CP DBeq_intro_BDP para} and has to be treated in the Besov-Bourgain-type spaces which is not done here.

\item The analogue of estimate \eqref{glob CP DBeq_intro_Para th CP imprv enrgy est_energy est} is still valid for $\alpha\geq2$ but some care is needed as the ``negative" H\"older spaces should be replaced by Zygmund spaces. In the real line case the standard Strichartz estimates give for $\alpha=2$, that is the Benjamin-Ono equation, the analogue of the well-posedness result of Burq and Planchon \cite{Burq08} and for $\alpha \in [2,3]$ the analogue of \cite{Herr07}.
\end{itemize}

We turn to the conjugation theorem for $\alpha\in [2,3]$.
\begin{theorem}\label{glob CP DBeq_intro_Para th conjg }
Consider two real numbers $ \alpha\in ]2,3[$, $s\in ]\frac{1}{2}+2-\alpha,+\infty[$. Then there exist $T>0$, $B>1$ and $r>0$ such that for all $u_0$ in the ball $\b(0,r)\subset  H^{s}(\d)$ there exists a unique $u\in C([0,T],H^s(\d))$ solving the Cauchy problem:
\begin{equation}\label{glob CP DBeq_intro_Para th conjg_eq1}
\partial_t u+T^{B,b}_{u}\partial_x u +\partial_x \abs{D}^{\alpha-1}u=0,\ u_0\in H^s.
\end{equation}
The flow map $v_0 \mapsto v$ is continuous from $\b(0,r)$ to $C([0,T],H^s(\d))$.

Moreover there exists a paradifferential operator $T_{p(u)}$ of order $0$ with \[p\in C\prt{[0,T],C^{2} \Gamma_1^{0}(\d)} \text{ and } \partial^{1+j}_t p \in C\prt{[0,T],C^{2} \Gamma_0^{(j+1)\alpha-1}(\d)}, \] where $C^{2}\Gamma_0^1$ are the symbol classes with limited  regularity in the frequency variable  defined in \ref{glob CP DBeq_sec impl constr symb_th constr comm_proof_def Banach scale symb} such that
 \begin{equation}\label{glob CP DBeq_intro_Para th conjg_eq2}
\partial_t \left[T^{B,b}_{e^{iT^{B,b}_{p(u)}}}u\right]+ \partial_x \abs{D}^{\alpha-1}\left[T^{B,b}_{e^{iT_{p(u)}}} u\right]=T^{\frac{B^2}{2B+1},b}_{R(u)}u,
\end{equation}
with $R\in C\prt{[0,T],C^{1} \Gamma_0^{0}(\d)}$.
\end{theorem}

The method used here can be pushed to prove that $p$ is in $C^k\Gamma_1^0(\d)_{k\in \n}$ for all $k \in \n$, with a smaller radius $r_k>0$ in Theorem \ref{glob CP DBeq_intro_Para th conjg } but without a lower bound on $r_k$. We chose to present the  computation showing $p\in C^{2}\Gamma_1^0(\d)$, which is the minimal regularity required for the definition of $e^{iT_{p(u)}}$ and $R(u)$.

For the ``paralinearised" KdV equation this gives local well-posedness in $H^{-\frac{1}{2}}(\d)$ which is the analogue of the result proven in \cite{Colliander03} for the full KdV equation.

\subsection{Strategy of the proof}
\paragraph*{\bf{The case $\alpha<2$}:} The starting point to prove the key apriori estimate \eqref{glob CP DBeq_intro_Para th CP imprv enrgy est_energy est} is  commuting $\D^s$ with the equation:
\[
\partial_t \D^su+i\frac{T_{ u\xi}+\left(T_{ u\xi}\right)^*}{2} \D^s u+\partial_x \abs{D}^{\alpha-1}\D^su=[T_u\partial_x,\D^s]u-i\frac{T_{ u\xi}-\left(T_{ u\xi}\right)^*}{2}\D^s u,
\]
where for an operator $A$, we write $A^*$ for it's $L^2$ adjoint. Thus defining $v=\D^su$ we get:
\[
\partial_t v+i\frac{T_{ u\xi}+\left(T_{ u\xi}\right)^*}{2} v+\partial_x \abs{D}^{\alpha-1}v=[T_u\partial_x,\D^s]\D^{-s}v-i\frac{T_{ u\xi}-\left(T_{ u\xi}\right)^*}{2} v.
\]
All of the terms in the right hand side are bounded in $L^2$ under the control of $\norm{\partial_x u}_{L^\infty}$. To improve upon this, following \cite{Ifrim17} we make a normal form transform: 
\[
w=v+B(v,v),
\]
where $B$ is a bi-linear operator of the form:
$$B(u,v)\sim D^{1-\alpha}R(u,v),$$
where $R(u,v)$ is the residual term generated in the paraproduct decomposition.
Thus we are brought to study an $L^2$ estimate on:
\[
\partial_t v+i\frac{T_{ u\xi}+\left(T_{ u\xi}\right)^*}{2} v+\partial_x \abs{D}^{\alpha-1}v=D^{2-\alpha}T(u,u,v),
\]
where $T$ is a bounded tri-linear operator. This is the exact analogue of the equation obtained in \cite{Ifrim17} after the normal form step.  Now we are looking to gauge transform the term $i\frac{T_{ u\xi}+\left(T_{ u\xi}\right)^*}{2}$ out of the equation. More precisely we are looking for an operator $A$ such that:
\begin{multline*}
\partial_t v+i\frac{T_{ u\xi}+\left(T_{ u\xi}\right)^*}{2} v+\partial_x \abs{D}^{\alpha-1}v=D^{2-\alpha}T(u,u,v)\\
\Rightarrow \partial_t v+\partial_x A^{-1}\abs{D}^{\alpha-1}Av=AD^{2-\alpha}T(u,u,v)+R_{2-\alpha},
\end{multline*}
where $A$ is a unitary operator modulo, at least, an $\alpha$-regularizing operator, $R_{2-\alpha}$ is of order $2-\alpha$ with $\RE(R_{2-\alpha})$ of order 0.

To find $A$ we follow our construction in \cite{Ayman20} and define $A=e^{iT_p}$, where $e^{i\tau T_p}$ is the flow of a hyperbolic paradifferential equation of the form:
\[
\partial_\tau e^{i\tau T_p} h_0-e^{i\tau T_p} h_0=0, \ e^{i 0 T_p}=Id.
\]
Using our result on the Baker-Campbell-Hausdorff formula for this type flow proved in \cite{Ayman20}, we are looking for $p$ such that:
\[
[e^{iT_p},\partial_x \abs{D}^{\alpha-1}]=e^{i T_p}i\frac{T_{ u\xi}+\left(T_{ u\xi}\right)^*}{2}+R.
\]

It's here where the two case $\alpha<2$ and $\alpha\geq 2$ have to be treated differently. Indeed, on one hand by Proposition \ref{BCH formula_def para hyperbolic flow_rem on prop com-conjg}
\[
[e^{iT_p},\partial_x \abs{D}^{\alpha-1}]=e^{i T_p}\ {}^c [\partial_x \abs{D}^{\alpha-1}]_1^p,
\]
where ${}^c [\partial_x \abs{D}^{\alpha-1}]_1^p$ belongs to a symbol class of the form $ L^{\infty}_*S_{\min(\alpha-1,1),2-\alpha}$. And on the other hand $$i\frac{T_{ u\xi}+\left(T_{ u\xi}\right)^*}{2}\in \Gamma^1_0=L^{\infty}_*S_{1,0},$$ thus for $\alpha<2$ there is no hope to solve $$ {}^c [\partial_x \abs{D}^{\alpha-1}]_1^p=i\frac{T_{ u\xi}+\left(T_{ u\xi}\right)^*}{2},$$
and for $\alpha\geq 2$ it is possible through a local implicit function theorem, which we will show in the next paragraph.\\

In \cite{Ayman20}, choosing $p=\frac{\xi\abs{\xi}^{1-\alpha}}{\alpha}U$, where $U$ is a primitive of $u$, was enough to have $R$ as an $\alpha-1$ regularizing operator when $s>\frac{3}{2}$ which gave us the desired result on the flow map regularity at this threshold. At our threshold of regularity where we only control the regularity of the solutions in $W^{2-\alpha}_x$ this would give $R$ as an operator of order $1$, that is there is no apparent gain that comes from this transformation.\\
To remedy this the idea is to construct $p$ implicitly. Indeed using the stability of paradifferential operators by commutation with $e^{iT_p}$ proved in \cite{Ayman20}, we solve the problem approximately using the ellipticity of $\xi\abs{\xi}^{\alpha-1}$, we show that  we can fully solve the first term in  the Baker-Campbell-Hausdorff expansion of ${}^c [\partial_x \abs{D}^{\alpha-1}]_1^p$, that is we solve:
 $$[iT_{p},T_{i\xi\abs{\xi}^{\alpha-1}}]=i\frac{T_{ u\xi}+\left(T_{ u\xi}\right)^*}{2}.$$

This amounts to right inverting a linear operator in the Fr\'echet space of  paradifferential operators. The problem is first reduced to a standard linear inversion in the scale of Banach spaces defining the Fr\'echet space of  paradifferential operators. Then using an explicit approximate parametrix, given by the usual Cole-Hopf choice of gauge transformation, a careful choice of cut-off functions studied in \cite{Ayman18} and symbolic calculus we show that a Neumann series can be carried out to correct the right parametrix into a right inverse in one Banach space in the scale. We then use a  bootstrap argument to propagate the regularity to the hole scale of Banach spaces and thus the Fr\'echet space of  paradifferential operators.

Getting back to the equation, we use the fact that $i\frac{T_{ u\xi}+\left(T_{ u\xi}\right)^*}{2}$ and $T_{i\xi\abs{\xi}^{\alpha-1}}$ are $L^2$ skew-adjoint to ensure that $p$ can be chosen $L^2$ self-adjoint. In doing so we construct an $L^2$ unitary operator $A=e^{iT_p}$ such that:
\[
[e^{iT_p},T_{i\xi\abs{\xi}^{\alpha-1}}]=e^{iT_p}\big(T_{i\xi}T_{ u}-T_{\frac{\partial_x u}{2}}\big)e^{-iT_p}+\underbrace{\int^1_0e^{i(1-r)T_p} [T_{ip},T_{u\partial_x+\frac{\partial_x u}{2}}]e^{i(r-1)T_p}dr}_{(1)}.
\]
Crudely at our threshold of regularity using symbolic calculus $(1)$ seems to be of order $1+2-\alpha$ which seems worse than what we had before, except in the case of the Benjamin-Ono equation, that is $\alpha=2$, indeed in that case:
\[
p(x,\xi)=\frac{\op(\frac{1}{D})P_{\geq b}(D)u}{2} \text{ for } \xi \geq 0,
\]
and, 
\[
\bigg[T^{B',b}_{ip},iT^{B',b}_{\sigma^{B,b}_{u\xi}+(\sigma^{B,b}_{u\xi})^*}\bigg]=T_{iu^2} \text{ for } \xi \geq 0,
\]
this ``exceptional" algebraic cancellation in the commutator is due to the $\partial_\xi p=0$ which does not occur for fractional $\alpha \neq 2$. This difficulty was noted in \cite{Herr10} and the proposed solution was to use a gauge transform with indeed $\partial_\xi p=0$ to eliminate only the lowest frequency terms in $u\partial_x u$, that is $P_0(D)u \partial_x u$, and treat the remainder terms in carefully chosen function spaces with frequency dependent time localisation. Inspired by this idea one can show that with the paradifferential setting developed here, the problem can be indeed reduced to a choice of $p$ such that $\partial_\xi p=0$, which will amount to a simple approximation of the symbol $p$ by step functions in the frequency variable $\xi$. But the careful use of the normal form transform combined with the implicit construction of $p$ ensuring self adjointness gives the following key cancellation which permits us to avoid this extra technical approximation step. Indeed using the identity:
\[
[B,C]^*=[C^*,B^*],
\]
 we see that $(1)$ is actually $L^2$ skew-adjoint, which gives the desired $L^2$ estimate after suitable control of the residual terms. 
 
\paragraph*{\bf{The case $\alpha\geq 2$}:}
The starting point of the problem is to try to completely conjugate the $T_{ u}\partial_x$ term away, that is we are looking for an operator $A$ such that
\[
\partial_t u+T_{ u}\partial_x u+T_{i\xi \abs{\xi}^{\alpha-1}}u=0\\
\Rightarrow \partial_t u+ A^{-1}T_{i\xi \abs{\xi}^{\alpha-1}}Av=0.
\]
A quick frequency localisation analysis shows that this can not be done completely using paradifferential operators, a residual term always appears, which essentially encodes that the all of Sobolev norms can not be exactly conserved or that the system is not exactly integrable.

Using the ellipticity of $\xi\abs{\xi}^{\alpha-1}$ and the paradifferential setting constructed we show that: $$p \mapsto T_{{}^c [i\xi\abs{\xi}^{\alpha-1}]_1^p}$$ is indeed locally surjective around $0$, which the key technical result we prove in Theorem \ref{glob CP DBeq_sec impl constr symb_thm constr gauge trsf}. This is a non trivial problem, equivalent to solving a nonlinear ODE in the Fr\'echet space of paradifferential symbols. Such an ODE is not generally well posed and to solve such a problem one usually has to look at a Nash-Moser type scheme\footnote{Such  a scheme can indeed be carried out here thanks to the tame estimates in Remark \ref{BCH formula_def para hyperbolic flow_rem on tame seminorm est} but we show that this can be avoided here.}. In our case the choice of paradifferential setting, inspired by by H\"ormander's \cite{Hormander90}, is shown to be stable by the gauge transformation in Proposition \ref{BCH formula_def para hyperbolic flow_prop com-conjg}. Thus We show that the problem can be reduced to a standard implicit function theorem combined with bootstrap argument in order to insure propagation of regularity. The bootstrap trick is the analogue of the one used in the Picard fixed point theorem depending on a parameter.

Thus we get,
\[T_{e^{iT_p}}\partial_tu+T_{i\xi\abs{\xi}^{\alpha-1}}T_{e^{iT_p}}u=T_{e^{iT_p}}\partial_x Res^1(u).
\]

Thus compared to first gauge transform we have to treat the term $\partial_t e^{iT_p}$ and $Res(u)$. At our threshold of regularity $\partial_t e^{iT_p}$ is still of order $1$ thus we don't have a gain on the order of the operator. One idea is to iterate the gauge transform to eliminate those time derivatives, that is construct $e^{iT_{p_n}}$ that eliminate the terms $(\partial_t e^{iT_{p_{n-1}}})$ at each step. While this schemes certainly works, we show that the problem can be solved in a ``cleaner" fashion, using a a simple application of a linear Nash-Moser scheme and again profiting from the para-linear setting to show that the map
$$p' \mapsto -T_{e^{-iT_{p'}}}T_{\partial_t e^{iT_{p'}}}+ T_{{}^c [i\xi\abs{\xi}^{\alpha-1}]_1^{p'}}$$ is also locally surjective around $0$. Thus we get,
\[\partial_t T_{e^{iT_p'}}u+T_{i\xi\abs{\xi}^{\alpha-1}}T_{e^{iT_p'}}u=T_{R(u)}u,
\]
where $R(u)$ is a residual para-differential term and is of order $0$.

\begin{remark}
The continuity of operators of type $T_{e^{i\tau p}}$ on Zygmund space with loss of derivatives were studied by E. Stein \cite{Stein93} and by G. Bourdaud in \cite{Bourdaud88}. In this paper we need explicit estimates taking into play the exact symbol semi-norms. For this we give a complete study of the continuity of paradifferential operators defined by symbols in these type of ``exotic" symbol classes in Appendix \ref{Cont of lim reg exotic sym}. Our proofs follow the same lines and methods presented in \cite{Stein93,Taylor91,Metivier08}. 
\end{remark}
\subsection{Acknowledgement}
I would like to express my sincere gratitude to my thesis advisor Thomas Alazard. I would also like to thank J-C. Saut for introducing me to this very interesting problem and the time and numerous discussions we had that helped me understand the problem.  
\section{Baker-Campbell-Hausdorff formula: composition and commutator estimates}\label{BCH formula}
We will start by giving the propositions defining the operators used in the gauge transforms and the symbolic calculus associated to them. All the main theorems are proved in \cite{Ayman20} and we will follow the same presentation, though we make a couple of more precise estimates on the semi-norms used, when we do so a proof is written. 
\begin{notation}
We will essentially compute the conjugation and commutation of operators with a flow map which naturally bring into play Lie derivatives that is commutators, thus we introduce the following notation for commutation between operators:
\[
\mathfrak{L}^0_a b=b, \ \mathfrak{L}_a b=[a,b]=a\circ b -b \circ a,\
\mathfrak{L}^2_a b=[a,[a,b]] , \ \mathfrak{L}^k_a b=\underbrace{[a,[\cdots,[a,}_{\text{k times}}b]]\cdots].
\] 

In the following propositions the variable $t\in [0,T]$ is the generic time variable that appear all through the paper and a new variable $\tau \in \r$ will be used and they should not be confused. 

We also need to recall the definition of the adjoint of a paradifferential symbol $p$ that we write $p^*$ and is given in \eqref{paracomposition_Notions of microlocal analysis_Paradifferential Calculus_definition adjoint para}.
\end{notation}
We start with the proposition defining the flow map and its standard properties. 
\begin{proposition}\label{BCH formula_def para hyperbolic flow}
Consider two real numbers $\delta \leq 1$, $s\in \r$ and a symbol $p\in \Gamma^{\delta}_0(\d)$ such that:
\[
\IM(p)=\frac{p-p^*}{2i}\in \Gamma^{\tilde{\delta}}_0(\d), \text{ with } \tilde{\delta}\leq 0.
\]
The following linear hyperbolic equation is well posed on $\r$:
\begin{equation}\label{BCH formula_def para hyperbolic flow_eq CP}
\begin{cases} 
\partial_\tau h-iT_{p} h=0, \\
h(0,\cdot)=h_0(\cdot)\in H^s.
\end{cases}
\end{equation}
For $\tau \in \r$, define $e^{i\tau T_p}$ as the flow map associated to \eqref{BCH formula_def para hyperbolic flow_eq CP} that is, 
\begin{align}\label{BCH formula_def para hyperbolic flow_eq def flow}
e^{i\tau T_p}:&H^s(\d) \rightarrow H^s(\d)\nonumber\\
&h_0 \mapsto h(\tau,\cdot).
\end{align}
Then for $\tau \in \r$ we have, 
\begin{enumerate}
\item $e^{i\tau T_p}\in \mathscr{L}(H^s(\d))$ and,
$$ \norm{e^{i\tau T_p}}_{H^s \rightarrow H^s}\leq e^{C\abs{\tau}  M_0^{0}(\IM(p))}. $$
\item $$iT_{p}\circ e^{i\tau T_p}=e^{i\tau T_p} \circ i T_{p}, \ e^{i(\tau+\tau') T_p}= e^{i\tau T_p}e^{i\tau' T_p}.$$
\item $e^{i\tau T_p}$ is invertible and,
 $$(e^{i\tau T_p})^{-1}=e^{-i\tau T_p}.$$ 
 Moreover the $L^2$ adjoint of $e^{i\tau T_p}$ verifies:
  $$(e^{i\tau T_p})^*=e^{-i\tau (T_p)^*}=e^{-i\tau T_p}+R,$$ 
  where $R$ is a $\tilde{\delta}$ regularizing operator and $e^{i\tau (T_p)^*}$ is the flow generated by the Cauchy problem:
  \begin{equation}\label{BCH formula_def para hyperbolic flow_CP def flow conjugate}
\begin{cases} 
\partial_\tau h-i(T_p)^* h=0, \\
h(0,\cdot)=h_0(\cdot)\in H^s(\d).
\end{cases}
\end{equation}
\item Taking a different symbol $\tilde{p}$ verifying the same hypothesis as $p$  we have:
\begin{equation}\label{BCH formula_def para hyperbolic flow_est diff flow}
   \norm{[e^{i\tau T_p}-e^{i\tau\tilde{p}}] h_0}_{H^s }\leq C\abs{\tau} e^{C\abs{\tau}  M_0^{0}(\IM(p),\IM(\tilde{p})) }M^\delta_0(p-\tilde{p})\norm{h_0}_{H^{s+\delta}}.
   \end{equation}
\end{enumerate}
\end{proposition}
\begin{proof}
In \cite{Ayman20} we worked with $p\in \Gamma_1^\delta$ and the continuity of $e^{i\tau T_p}$, that is point $(1)$, was proved through an energy estimate. Looking closely to the energy estimate we see that that we only need to control the semi norm $M_0^0(\RE(ip))$ and with the hypothesis on $p$ we get the more precise result.
\end{proof}

Under the hypothesis $p\in \Gamma^{\delta}_1(\d)$ and $p$ real valued we have $\tilde{\delta}\leq\delta-1$ by symbolic calculus and the inequality:
\[
M_0^{\delta-1}(\IM(p))\leq C M_0^{\delta-1}(\partial_\xi \partial_x p), \ \tilde{\delta}\leq 0,
\]
which automatically verified the desired conditions in \cite{Ayman20}.

Later on we will need to study the continuity of $e^{i\tau T_p}$ on H\"older/Zygmund spaces. This a non trivial result, indeed hyperbolic flows are not in general continuous on $L^p$ spaces for $p\neq 2$, as it is for example the case for the Schr\"odinger equation, or equations of the form $\partial_t h+i \abs{D}^{\alpha}h=0, \alpha \neq 1$ that are not continuous on Zygmund spaces as shown in the Appendix of \cite{Alazard16}. 
To study the continuity of $e^{i\tau T_p}$ on H\"older/Zygmund we start by studying its symbol. First we recall the following Lemma we proved in \cite{Ayman20} adapting the classic result by Beals on pseudodifferential operators in \cite{Beals77} to the limited regularity setting.

\begin{lemma}\label{BCH formula_def para hyperbolic flow_lem Beals seminorm estimates adapted}
Consider an operator $A$ continuous from $\sr(\d)$ to $\sr'(\d)$ and let $a\in \sr'(\d\times \hat{\d})$ be the unique symbol associated to A (cf. \cite{Bony 13} for the uniqueness), that is, let $K$ be the kernel associated to $A$ then:
\[
u,v \in \sr(\d), (Au,v)=K(u\otimes v),\ \ a(x,\xi)=\fr_{y\rightarrow \xi}K(x,x-y).
\]
\begin{itemize}
\item If $A$ is continuous from $H^m$ to $L^2$, with $m\in  \r$, and $[\frac{1}{i}\frac{d}{dx},A]$ is continuous from $H^{m+\delta}$ to $L^2$ with $\delta<1$, then $(1+\abs{\xi})^{-m}a(x,\xi)\in L^\infty_{x,\xi}(\d \times \hat{\d})$ and we have the estimate:
\begin{equation}\label{BCH formula_def para hyperbolic flow_lem Beals seminorm estimates adapted_est in x}
\norm{(1+\abs{\xi})^{-m}a}_{L^\infty_{x,\xi}}\leq C_m \bigg[\norm{A}_{H^m \rightarrow L^2}+\norm{\bigg[\frac{1}{i}\frac{d}{dx},A\bigg]}_{H^{m+\delta} \rightarrow L^2}\bigg].
\end{equation}
\item If $A$ is continuous from $H^m$ to $L^2$, with $m\in  \r$, and $[ix,A]$ is continuous from from $H^{m-\rho}$ to $L^2$ with $\rho \geq 0$, then $(1+\abs{\xi})^{-m}a(x,\xi)\in L^\infty_{x,\xi}(\d \times \hat{\d})$ and we have the estimate:
\begin{equation}\label{BCH formula_def para hyperbolic flow_lem Beals seminorm estimates adapted_est in xi}
\norm{(1+\abs{\xi})^{-m}a}_{L^\infty_{x,\xi}}\leq C_m [\norm{A}_{H^m \rightarrow L^2}+\norm{[ix,A]}_{H^{m-\rho} \rightarrow L^2}].
\end{equation}
\end{itemize}
\end{lemma}

We can now define the limited-regularity symbol classes to which $e^{i\tau T_p}$ belongs.
\begin{definition}\label{BCH formula_def para hyperbolic flow_def nonreg symb}
Consider $s\in \r_+$, for $0\leq \delta, \rho <1$, we say:
\begin{equation}\label{BCH formula_def para hyperbolic flow_def nonreg symb_eq1}
p\in W^{s,\infty} S^m_{\rho,\delta}(\d) \iff 
\begin{cases}
 \norm{D^k_\xi p(\cdot,\xi)}_{W^{s,\infty}} \leq C_{k,0} \langle \xi \rangle^{m-\rho k}  \vspace*{0.2cm}\\
  \abs{D^{\lfloor s \rfloor+n}_x D^k_\xi p(x,\xi)} \leq C_{k,n} \langle \xi \rangle^{m-\rho k+(n+\lfloor s \rfloor-s)\delta}
 \end{cases}, (x,\xi)\in \d \times \tilde{\d},
\end{equation}
for $k\geq 0$ and $n\geq 1$. The best constants $C_{k,n}$ in \eqref{BCH formula_def para hyperbolic flow_def nonreg symb_eq1} define a family of seminorms denoted by ${}^{\rho,\delta}M^m_{n,s}(\cdot;k),(k,n)\in \n^2$ where $k$ is the number of derivatives we make on the frequency variable $\xi$ and $n$ is that in the $x$ variable. We also define the seminorm ${}^{\rho,\delta}M^m_s(\cdot)={}^{\rho,\delta}M^m_{0,s}(\cdot;1)$ and define  analogously $W^{s,\infty} S^m_{\rho,\delta}(\d^*\times  \tilde{\d})$. 

Motivated by Lemma \ref{BCH formula_def para hyperbolic flow_lem Beals seminorm estimates adapted} we introduce the following family of seminorms:
\begin{multline*}
{}^{\rho,\delta}H^m_{0,s}(p;k)=\sum_{l=0}^{\lfloor s \rfloor}\sum^k_{j=0}\norm{\mathfrak{L}^j_{ix}\mathfrak{L}^l_{\frac{1}{i}\frac{d}{dx}}\op(p)}_{H^m \rightarrow H^{j\rho}}\\+\sup_{n\in \n}2^{n(s-\lfloor s \rfloor)}\sum^k_{j=0}\norm{\mathfrak{L}^j_{ix}\mathfrak{L}^{\lfloor s \rfloor}_{\frac{1}{i}\frac{d}{dx}}[P_{n}(D)\op(p)] }_{H^m \rightarrow H^{j\rho}},
\end{multline*}
where $P_n(D)$ is applied to $p$ in the $x$ variable. And for $n\geq 1$ 
\[
{}^{\rho,\delta}H^m_{n,s}(p;k)={}^{\rho,\delta}H^m_{0,s}(p;k)+\sum_{l=1}^{n}\sum^k_{j=0}\norm{\mathfrak{L}^j_{ix}\mathfrak{L}^{\lfloor s \rfloor+l}_{\frac{1}{i}\frac{d}{dx}}\op(p)}_{H^m \rightarrow H^{j\rho-(l+\lfloor s \rfloor-s)\delta}}.
\]
Then ${}^{\rho,\delta}H^m_{n,s}(\cdot;k)_{(n,k)\in \n^2}$ induces an equivalent Fr\'echet topology to ${}^{\rho,\delta}M^m_{n,s}(\cdot;k)_{(n,k)\in \n^2}$ on $W^{s,\infty} S^m_{\rho,\delta}$.
\end{definition}

In Stein's \cite{Stein93}, such symbols are called ``exotic", their continuity on different $L^p$ spaces is completely studied but only in the regular case and without explicit estimates depending on the semi-norms. To make such estimates explicit we have given a full proof of such continuity theorems in Appendix \ref{Cont of lim reg exotic sym}.

\begin{proposition}\label{BCH formula_lem est const lie derv_prop Atau symb}
Consider two real numbers $\delta < 1$, $\rho \geq 0$ and a symbol $p\in \Gamma^{\delta}_\rho(\d)$ such that:
\[
\IM(p)=\frac{p-p^*}{2i}\in \Gamma^{\tilde{\delta}}_0(\d), \text{ with } \tilde{\delta}\leq 0.
\]
 Let $e^{i\tau T_p},\tau \in \r$ be the flow map defined by Proposition \ref{BCH formula_def para hyperbolic flow}, then there exists  a symbol $e_\otimes^{i\tau p} \in W^{\rho,\infty}S^0_{1-\delta,\delta}(\d^*\times  \tilde{\d})$ such that:
\begin{equation}\label{BCH formula_lem est const lie derv_Atau symb eq}
e^{i\tau T_p}=\op(e_\otimes^{i\tau p}).
\end{equation}
Moreover we have the identity: 
\begin{equation}\label{BCH formula_lem est const lie derv_Atau symb ident 2}
e^{i\tau T_p}=T_{e^{i\tau p}}+\int_0^\tau e^{i(\tau-s)p}\big(T_{ip}T_{e^{is p}}-T_{ipe^{is p}} \big)ds.
\end{equation}
\end{proposition}
\begin{remark}\label{BCH formula_remark self contained}
In this paper we will only use the result for $\delta=0$ and thus the previous statement is simply a corollary of stability under composition of the class operators $a \in S^0_{1,1}$ such that it's $a^*\in S^0_{1,1}$ where $a^*$ is given by
$$
a^*(x,\xi)=\frac{1}{2\pi}\int_{\d\times \hat{\d}}e^{-iy.\eta} \bar{a}(x-y,\xi-\eta)dyd\eta,
$$ 
This a classic result first given by Bourdaud in \cite{Bourdaud88} and by H\"ormander in \cite{Hormander89}. We opted to keep the statements of the more general theorems of \cite{Ayman20} as we think they are important in understanding the larger picture and difficulties when preforming the gauge transformation and how to get over them.
\end{remark}
Combining the previous Proposition with Theorem \ref{Cont of lim reg exotic sym_thm for para symb} we get the following:
\begin{corollary}\label{BCH formula_Atau Zyg cont}
Consider two real numbers $\delta < 1$, $\rho \geq 0$ and a symbol $p\in \Gamma^{\delta}_\rho(\d)$ such that:
\[
\IM(p)=\frac{p-p^*}{2i}\in \Gamma^{\tilde{\delta}}_0(\d), \text{ with } \tilde{\delta}\leq 0.
\]

Then $e^{i\tau T_p}$ is continuous from $C_*^s$ to $C_*^{s-\frac{\delta}{2}}$ and from $W^{s+(\frac{1}{2}-\frac{1}{p})\delta,p}$ to $ W^{s,p}$ for $s>0$. Moreover we have the estimate:
\[\norm{e_\otimes^{i\tau p}}_{W^{s+(\frac{1}{2}-\frac{1}{p})\delta,p}\rightarrow W^{s,p}}\leq K \ {}^{1-\delta,\delta} M^{0}_{0}(e_\otimes^{i\tau p};1),\text{ and,}\]
\[\norm{e_\otimes^{i\tau p}}_{C^{s+\frac{1}{2}\delta}_*\rightarrow C^{s}_*}\leq K \  {}^{1-\delta,\delta}M^{0}_{0}(e_\otimes^{i\tau p};1).  \]
\end{corollary}

The key commutation and conjugation result is given by the following proposition, again we give the general result given in \cite{Ayman20} but Remark \ref{BCH formula_remark self contained} applies here too.

\begin{proposition}\label{BCH formula_def para hyperbolic flow_prop com-conjg}
Consider two real numbers $\delta < 1$, $\rho \geq 0$ and a symbol $p\in \Gamma^{\delta}_\rho(\d)$ such that:
\[
\IM(p)=\frac{p-p^*}{2i}\in \Gamma^{\tilde{\delta}}_0(\d), \text{ with } \tilde{\delta}\leq 0.
\]
Let $e^{i\tau T_p},\tau \in \r$ be the flow map defined by Proposition \ref{BCH formula_def para hyperbolic flow} and take a symbol $b\in \Gamma^{\beta}_{\rho}(\d),\beta \in \r$ then we have:
\begin{enumerate}
\setcounter{enumi}{4}
\item For $\rho \geq 1$, there exists $b_\tau^p \in  W^{\rho,\infty} S^\beta_{1-\delta,\delta}(\d)$ such that:
\begin{align}\label{BCH formula_def para hyperbolic flow_prop com-conjg_eq def cong}
e^{i\tau T_p} \circ T_b \circ e^{-i\tau T_p}&=\op(b_\tau^p).
\end{align}
Moreover we have the estimates:
\begin{equation}\label{BCH formula_def para hyperbolic flow_prop com-conjg_est symbolic calc conjg}
\norm{\op(b_\tau^p)-\sum_{k=0}^{\lceil \rho-1 \rceil}\frac{\tau^k}{k!}\mathfrak{L}^k_{iT_p}T_b}_{H^s\rightarrow H^{s-\beta-\lceil \rho \rceil \delta+\rho}}\leq C_\rho M_\rho^\beta(b) M_\rho^\delta(p)^{\lceil \rho \rceil},
\end{equation}
\begin{equation}\label{BCH formula_def para hyperbolic flow_rem on tame seminorm est_eq1}
{}^{1-\delta,\delta}H^\beta_\rho(b_\tau^p;k)\leq C_{\rho,k} e^{\tau CM_0^{0}(\IM(p))} [H^\beta_\rho(b;k)+H^\beta_\rho(b;k)H^{\delta}_\rho(p;k)], \ k\in \n,
\end{equation}
where $C_{\rho,k}$ is a constant depending only on $\rho$ and $k$.
\item There exists ${}^c b_\tau^p \in  W^{\rho-1,\infty} S^{\beta+\delta-1}_{1-\delta,\delta}(\d)$ such that:
\begin{align}\label{BCH formula_def para hyperbolic flow_prop com-conjg_eq def com}
[e^{i\tau T_p}, T_b]&=e^{i\tau T_p} \op({}^c b_\tau^p) \iff \op({}^c b_\tau^p)=T_b-\op(b_{-\tau}^p).
\end{align}
Moreover we have the estimates:
\begin{equation}\label{BCH formula_def para hyperbolic flow_prop com-conjg_est symbolic calc com}
\norm{ \op({}^c b_\tau^p)-\sum_{k=1}^{\lceil \rho-1 \rceil}(-1)^{k-1}\frac{\tau^k}{k!}\mathfrak{L}^k_{iT_p}T_b}_{H^s\rightarrow H^{s-\beta-\lceil \rho \rceil \delta+\rho}}\leq C_\rho M_\rho^\beta(b) M_\rho^\delta(p)^{\lceil \rho \rceil},
\end{equation}
\begin{equation}\label{BCH formula_def para hyperbolic flow_prop com-conjg_est seminorm com}
{}^{1-\delta,\delta}H^{\beta+\delta-1}_{\rho-1}({}^c b_\tau^p;k)\leq C_{\rho,k} e^{\tau CM_0^{0}(\IM(p))} [H^\beta_\rho(b;k)+H^\beta_\rho(b;k)H^{\delta}_\rho(p;k)], k\in \n,
\end{equation} 
where $C_{\rho,k}$ is a constant depending only on $\rho$ and $k$.
\end{enumerate}
The link between $b_\tau^p$ and ${}^c b_\tau^p$ is given by the following:
\[
\op({}^c b_\tau^p)=\int_0^\tau e^{irT_p} T_{\mathfrak{L}_{ip}b}  e^{-irT_p}dr=\int_0^\tau \op(\mathfrak{L}_{iT_p}T_b)_r^pdr,
\]
where $\mathfrak{L}_{ip}b$ is the paradifferential symbol associated to $\mathfrak{L}_{iT_p}T_b$ by Theorem \ref{paracomposition_Notions of microlocal analysis_Paradifferential Calculus_symbolic calculus para precised}.

\end{proposition}
\begin{remark}\label{BCH formula_def para hyperbolic flow_rem on prop com-conjg}
\begin{itemize}
\item It is important to notice that the main result of this proposition is the factorisation of the $e^{i\tau T_p}$ terms in \eqref{BCH formula_def para hyperbolic flow_prop com-conjg_eq def cong} and \eqref{BCH formula_def para hyperbolic flow_prop com-conjg_eq def com} where the right hand sides contain symbols in the usual classes. This was not apriori the case of the left hand sides containing $e^{i\tau T_p}$. In other words we study the stability of $\Gamma^m_\rho$ under the conjugation by $e^{i\tau T_p}$.

\item In the language of pseudodifferential operators, $T_{b_\tau^p}$ is the asymptotic sum of the series $(\frac{\tau^k}{k!}\mathfrak{L}^k_{iT_p}T_b)$ that is the Baker-Campbell-Hausdorff formal series. Though $T_{b_\tau^p}$ is not necessarily equal to this sum, for this sum need not converge.
\end{itemize}
\end{remark}

\begin{remark}\label{BCH formula_def para hyperbolic flow_rem on tame seminorm est}
We would like to note that in the special case $\delta=0$ we have the refined tame estimates for $k\in \n$:
\begin{equation}\label{BCH formula_def para hyperbolic flow_rem on tame seminorm est_eq1}
M^\beta_0(\partial^k_\xi b_\tau^p;0)\leq \sum^k_{j=0}\sum^j_{l=0} \binom{k}{j}\binom{j}{l}M^0_0(\partial^{k-j}_\xi e^{i\tau p}_\otimes;0)M^\beta_0(\partial^{j-l}_\xi b_\tau^p;0)M^0_0(\partial^l_\xi  e^{-i\tau p}_\otimes;0).
\end{equation}
We won't explicitly use the tameness in our proof as we avoid using a Nash-Moser type scheme but it is worth noting that implicitly it is this condition that ensures that the constructions in Section \ref{glob CP DBeq_sec impl constr symb} converge, for more details on the necessity of this condition we refer to the following complete and instructive article by Hamilton \cite{Hamilton82}. 
\begin{proof}
This is the consequences of the Leibniz formula combined with the  computation of $[ix,b^p_\tau]$:
\[
[ix,e^{i\tau T_p} T_{b}e^{-i\tau T_p}]= [ix,e^{i\tau T_p}]T_b e^{-i\tau T_p}+e^{i\tau T_p} [ix,T_b]e^{-i\tau T_p}+e^{i\tau T_p} T_b[ix,e^{-i\tau T_p}].
\]
\end{proof}
\end{remark}

The different Gateaux derivatives of the operators defined above are given by the following propositions.
\begin{proposition}\label{BCH formula_prop diff Atau param}
Consider two real numbers $\delta < 1$, $\rho \geq 0$, two symbols $p,p'\in \Gamma^{\delta}_{\rho}(\d)$ such that,
\[
\IM(p)=\frac{p-p^*}{2i}\in \Gamma^{\tilde{\delta}}_0(\d),\IM(p')=\frac{p'-p'^*}{2i}\in \Gamma^{\tilde{\delta}}_0(\d), \ \tilde{\delta}\leq 0.
\]
 Let $e^{i\tau T_p},e^{i\tau p'},\tau \in \r$ be the flow maps defined by Proposition \ref{BCH formula_def para hyperbolic flow}, then for $\tau \in \r$ we have:
\begin{equation}\label{BCH formula_prop diff Atau param_eq1}
 e^{i\tau T_p}-e^{i\tau T_{p'}}=\int_0^\tau e^{i(\tau-r) T_p}T_{ip'- p}e^{ir T_{p'}}dr.
\end{equation}
Another way to express this is with the Gateaux derivative of $p \mapsto e^{i\tau T_p}$ on the Fr\'echet space $\Gamma^{\delta}_{\rho}(\d)$ is given by:
\begin{equation}\label{BCH formula_prop diff Atau param_eq2}
D_p e^{i\tau T_p}(h)=\int_0^\tau e^{i(\tau-r) T_p}T_{ih}e^{ir T_p}dr.
\end{equation}

Moreover consider an open interval $I\subset \r$, and a symbols $p\in C^1(I,\Gamma^{\delta}_{\rho}(\d))$ such that for all $z\in I$:
\[
\IM(p(z))=\frac{p-p^*}{2i}\in \Gamma^{\tilde{\delta}}_0(\d),
\]
 Let $e^{i\tau T_p},\tau \in \r$ be the flow map defined by Proposition \ref{BCH formula_def para hyperbolic flow} then for $\tau \in \r, z\in I$ we have:
\begin{equation}\label{BCH formula_prop diff Atau param_eq3}
\partial_z e^{i\tau T_p}=\int_0^\tau e^{i(\tau-r) T_p}T_{i\partial_z p}e^{ir T_p}dr.
\end{equation}
\end{proposition}

\begin{proposition}\label{BCH formula_prop diff Atau conjg}
Consider two real numbers $\delta < 1$, $\rho >1$, $\rho \notin \n$, and two symbols $p,p'\in \Gamma^{\delta}_{\rho}(\d)$ verifying:
\[
\IM(p)=\frac{p-p^*}{2i}\in \Gamma^{\tilde{\delta}}_0(\d),\IM(p')=\frac{p'-p'^*}{2i}\in \Gamma^{\tilde{\delta}}_0(\d), \ \tilde{\delta}\leq 0.
\]
Let $e^{i\tau T_p},e^{i\tau T_{p'}},\tau \in \r$ be the flow maps defined by Proposition \ref{BCH formula_def para hyperbolic flow} and take a symbol $b\in \Gamma^{\beta}_{\rho}(\d)$ then for $\tau \in \r$ we have:
\begin{align}
\op(b^p_\tau)-\op(b^{p'}_\tau)&=\int_0^\tau e^{i(\tau-r) T_p} \mathscr{L}_{iT_{p-p'}}\op(b^{p'}_r) e^{i(r-\tau)T_p}dr \label{BCH formula_prop diff Atau conjg_eq1}\\
&=i\int_0^\tau \mathscr{L}_{T_{p}-\op(p'^{p}_{\tau-r})}\op((b^{p'}_r)^p_{\tau-r})dr.\label{BCH formula_prop diff Atau conjg_eq2}
\end{align}
Another way to express this is with the Gateaux derivative of $p \mapsto \op(b^p_\tau)$ on the Fr\'echet space $\Gamma^{\delta}_{\rho}(\d)$ is given by:
\begin{equation}\label{BCH formula_prop diff Atau conjg_eq3}
D_p\op(b^p_\tau)(h)=\int_0^\tau\mathscr{L}_{i \op(h^{p}_{\tau-r})}\op(b^{p}_{\tau})dr=\mathscr{L}_{i\int_0^\tau \op(h^{p}_{\tau-r})dr}\op(b^{p}_{\tau}).
\end{equation}
Writing,
$\op({}^c b^{p}_{\tau})=T_b-\op(b^{p}_{-\tau})$, and, $\op({}^c b^{p'}_{\tau})=T_b-\op(b^{p'}_{-\tau})$ we get:
\begin{align}
\op({}^c b^{p}_{\tau})-\op({}^c b^{p'}_{\tau})&=-\int_0^{-\tau} e^{-i(\tau+r)T_p} \mathscr{L}_{iT_{p-p'}}\op(b^{p'}_r) e^{i(\tau+r)T_p}dr \label{BCH formula_prop diff Atau conjg_eq4}\\
&=-\int_0^{-\tau} \mathscr{L}_{iT_{p}-\op(p'^p_{-\tau-r})}\op((b^{p'}_r)_{-\tau-r}^p )dr.\label{BCH formula_prop diff Atau conjg_eq5}
\end{align}
\begin{equation}\label{BCH formula_prop diff Atau conjg_eq6}
D_p\op({}^c b^{p}_{\tau})(h)=-\int_0^{-\tau}\mathscr{L}_{i \op(h^p_{-\tau-r})}\op(v^p_{-\tau})dr=-\mathscr{L}_{i\int_0^{-\tau} \op(h^p_{-\tau-r})dr}\op(v^p_{-\tau}).
\end{equation}
\end{proposition}

We now study the composition of two different flows.
\begin{theorem}\label{BCH formula_prop Atau comp}
Consider two real numbers $\delta < 1$, $\rho \geq 0$, two symbols $p,p'\in \Gamma^{\delta}_{\rho}(\d)$ such that,
\[
\IM(p)=\frac{p-p^*}{2i}\in \Gamma^{\tilde{\delta}}_0(\d),\IM(p')=\frac{p'-p'^*}{2i}\in \Gamma^{\tilde{\delta}}_0(\d), \ \tilde{\delta}\leq 0.
\]
 Then for $\tau \in \r$ we have:
\[
e^{i\tau T_p}e^{i\tau T_{p'}}=e^{i\tau T_p+i\int_0^r\op(p')^p_rdr}. 
\]
\end{theorem}
Strictly speaking we only presented flows that were generated by operators independent of the $\tau$ variable which is not the case of $e^{i\tau T_p+i\int_0^r\op(p')^p_rdr}$. We did so to avoid burdening the presentation, one can see all the results of this section can in verbatim be generalised to operators with Lipschitz dependence on $\tau$ by the usual Cauchy-Lipschitz theorem. 

\begin{proof}
Fix $h_0\in H^s,s\in \r$ and compute:
\[
\partial_{\tau}[e^{i\tau T_p}e^{i\tau T{p'}}h_0]=-iT_p[e^{i\tau T_p}e^{i\tau T_{p'}}h_0]-i[e^{i\tau T_{p}} T_{p'}e^{i\tau T_{p'}}h_0],
\]
thus by Proposition \ref{BCH formula_def para hyperbolic flow_prop com-conjg},
\[
\partial_{\tau}[e^{i\tau T_{p}}e^{i\tau T_{p'}}h_0]=-i(T_p+\op(p')^p_\tau)[e^{i\tau T_p}e^{i\tau T_{p'}}h_0],
\]
and $e^{i\tau T_p}e^{i\tau T_{p'}}h_0(0,\cdot)=h_0(\cdot)$ which gives the desired result.
\end{proof}

\section{Implicit construction of symbols}\label{glob CP DBeq_sec impl constr symb}
In this section with the different theorems permitting the construction of the gauge transforms used in the subsequent sections.
\begin{theorem}\label{glob CP DBeq_sec impl constr symb_th constr comm}
Consider two real numbers $\alpha\geq 1$, $\beta\in \r$ and a symbol $a\in \Gamma_0^\beta(\d)$. Then there exist $B>1$ and a symbol $p\in\Gamma_1^{\beta+1-\alpha}(\d)$ such that, 
\begin{equation}\label{glob CP DBeq_sec impl constr symb_th constr comm_para eq}
\sigma^{B,b}_{p}\otimes i\xi \abs{\xi}^{\alpha-1}-i\xi \abs{\xi}^{\alpha-1}\otimes\sigma^{B,b}_{p} =\sigma^{B,b}_{a},
\end{equation}
where $\otimes$ is the symbol product defined formally by:
\[
\op(p)\circ \op(q)=\op(p\otimes q),  \text{ where }\]
\[
p\otimes q(x,\xi)=\frac{1}{2\pi}\int_{\d \times \hat{\d}}e^{i(x-y).(\xi-\eta)} p(x,\eta)q(y,\xi)dyd\eta,
\]
and $\sigma^{B,b}_{\cdot}$ is a cutoff defining paradifferential operators (cf Definition \ref{paracomposition_Notions of microlocal analysis_Paradifferential Calculus_def para op}).

Moreover we have the estimates:
\begin{equation}\label{glob CP DBeq_sec impl constr symb_th constr comm_est costr comm}
M^{\beta+1-\alpha}_0(\partial_x \sigma^{B,b}_{p};0)\leq \frac{M^{\beta}_0(\sigma^{B,b}_{a};0)}{B[1-(1-\frac{1}{B})^\alpha]},
\end{equation}
\begin{align}\label{glob CP DBeq_sec impl constr symb_th constr comm_est costr comm derv xi}
M^{\beta-\alpha}_0(\partial_\xi\partial_x \sigma^{B,b}_{p};0)&\leq \frac{M^{\beta-1}_0(\partial_\xi\sigma^{B,b}_{a};0)}{B[1-(1-\frac{1}{B})^\alpha]}\\
&+\alpha\frac{(1+\frac{1}{B})^{\alpha-1}-1}{1-(1-\frac{1}{B})^\alpha}M^{\beta+1-\alpha}_0(\partial_x \sigma^{B,b}_{p};0).\nonumber
\end{align}
\end{theorem}
The choice of the same cut-off parameters in the right hand side and left hand side of \eqref{glob CP DBeq_sec impl constr symb_th constr comm_para eq} is not immediate, indeed by the general rule  of composition of paradifferential operators given  in Proposition \ref{paracomposition_Notions of microlocal analysis_Paradifferential Calculus_symbolic calculus para precised}, the cut-off on the left hand side is given by $B\star B=\frac{B^2}{2B+1}>B$. But in the specific case where one of the operators is a Fourier multiplier we have this refined property where the cut-off on the left hand side is indeed given by $B$.
\begin{proof}
First the case $\alpha=1$ has the immediate solution with the choice of $p$ as the  primitive of $\sigma^{B,b}_{a}$ in the $x$ variable. Henceforth we suppose $\alpha>1$.

Formally $p$ should be given explicitly given by
\begin{equation}\label{glob CP DBeq_sec impl constr symb_th constr comm_proof_explicit formula symb}
\fr_{x\rightarrow\eta }(p)(\eta,\xi)=\frac{\psi^{B,b}(\eta,\xi)\fr_{x\rightarrow\eta }(a)(\eta,\xi)}{i\prt{\abs{\xi}^{\alpha-1}\xi-(\eta+\xi)\abs{\xi+\xi}^{\alpha-1}}},
\end{equation}
where $\psi^{B,b}$ is the cut-off defining the regularisation $\sigma^{B,b}_{\cdot}$. To better understand the denominator we recall Lemma $2.1$ from \cite{Molinet15} on the resonance function.
\begin{lemma}[Lemma $2.1$ from \cite{Molinet15}]\label{glob CP DBeq_sec impl constr symb_resonnance function lem}
Let us define the resonance function of order $2$ associated with the weakly dispersive Burgers equation:
\begin{equation}\label{glob CP DBeq_sec impl constr symb_resonnance function def}
\Omega_{\alpha}(\xi_1,\xi_2)=(\xi_1+\xi_2)\abs{\xi_1+\xi_2}^{\alpha-1}-\xi_1\abs{\xi_1}^{\alpha-1}-\xi_2\abs{\xi_2}^{\alpha-1}, \text{ for } \xi_1,\xi_2 \in \r.
\end{equation}
For $\xi_1,\xi_2,\xi_3 \in \r$ define the quantities $\abs{\xi_{max}}\geq \abs{\xi_{med}}\geq \abs{\xi_{min}}$  to be the maximum, median, minimum of $\abs{\xi_1},\abs{\xi_2}$ and $\abs{\xi_3}$ respectively.

Then for $\alpha>1$ and $\xi_3=-(\xi_1+\xi_2)$,
\begin{equation}
\abs{\Omega(\xi_1,\xi_2)}\sim \abs{\xi_{min}}\abs{\xi_{max}}^{\alpha-1}.
\end{equation}
\end{lemma}
Combining \eqref{glob CP DBeq_sec impl constr symb_th constr comm_proof_explicit formula symb} with Lemma \ref{glob CP DBeq_sec impl constr symb_resonnance function lem} one can make direct estimates on $p$ to show it belongs to the desired symbol classes. While this approach definitely works it is limited to this linear in $p$ case, we give another proof following a more robust fixed point scheme that will generalize nicely to the non linear in $p$ case, that is the $[e^{iT_{p}},\partial_x\abs{D}^{\alpha-1}]$ problem.

We start by defining the scale of Banach spaces that define the Fr\'echet space of Paradifferential operators.
\begin{definition}\label{glob CP DBeq_sec impl constr symb_th constr comm_proof_def Banach scale symb}
Given $m \in \r$, $ k \in \n$ and $\w\subset \sr'$ a Banach space. Define $C^k\Gamma^m_\w(\d \times \hat{\d} \setminus \b(0,R))$ as the space of locally bounded functions $a(x,\xi)$ defined on $\d \times \hat{\d} \setminus \b(0,R)$, which are $C^k$ with respect to $\xi$ and such that, for all $j \leq k$ and for all $\xi \geq R$, the function $x \mapsto \partial^\alpha_\xi a(x,\xi)$ belongs to $\w$ and there exists a constant $C_{k}$ such that:
\begin{equation}\label{glob CP DBeq_sec impl constr symb_th constr comm_proof_def Banach scale symb_eq growth xxi cond}
\text{ for } \abs{\xi}\geq R, j\leq k, \norm{\partial^j_\xi a(.,\xi)}_{\w}\leq C_{k} (1+\abs{\xi})^{m-\abs{\alpha}}. 
\end{equation}
The space $C^k\Gamma^m_\w(\d \times \hat{\d} \setminus \b(0,R))$ is equipped with it's natural Banach space topology induced by the best constant $C_{k}$. When $\w=W^{\rho,\infty}$, the best constant is the seminorm $M_\rho^\beta(\cdot;k)$.
\end{definition}
We define:
\[
\psi^{B,b}\bigg(\Gamma^m_\w(\d)\bigg)=\set{\sigma^{B,b}_p,p\in \Gamma^m_\w(\d)},
\]
\[
\psi^{B,b}\bigg(C^k\Gamma^m_\w(\d \times \hat{\d} \setminus \b(0,R))\bigg)=\set{\sigma^{B,b}_p,p\in C^k\Gamma^m_\w(\d \times \hat{\d} \setminus \b(0,R))},
\]
equipped with their natural Fr\'echet and Banach topologies induced by the continuity of the map $p\mapsto\sigma^{B,b}_p$.
Now we reinterpret Theorem \ref{glob CP DBeq_sec impl constr symb_th constr comm} by introducing the linear operator:
\begin{align*}
   L:&C^k\Gamma^{\beta+1-\alpha}_1(\d \times \hat{\d} \setminus \b(0,b)) \rightarrow  \psi^{B,b}\bigg(C^k\Gamma^\beta_0(\d \times \hat{\d} \setminus \b(0,b))\bigg) \\
  &p \mapsto \sigma^{B,b}_{p}\otimes i\xi \abs{\xi}^{\alpha-1}-i\xi \abs{\xi}^{\alpha-1}\otimes\sigma^{B,b}_{p}. 
  \end{align*}
The proof will then proceed in two steps, we first prove that $L$ has a right inverse on the Banach space $C^0\Gamma^{\beta+1-\alpha}_1(\r \setminus \b(0,b))$ and then we prove propagation of regularity in the frequency variable $\xi$ for solutions of the equation \eqref{glob CP DBeq_sec impl constr symb_th constr comm_para eq}.

To construct a right inverse for $L$ the key idea here is simply that a right hand parametrix is given by the standard Cole-Hopf gauge transform:
\begin{align*}
   E_{approx}:&C^k\Gamma^\beta_0(\d \times \hat{\d} \setminus \b(0,b))\rightarrow  \psi^{B,b}\bigg(C^k\Gamma^{\beta+1-\alpha}_0(\d \times \hat{\d} \setminus \b(0,b))\bigg) \\
  &a(x,\xi) \rightarrow \frac{\abs{\xi}^{1-\alpha}}{\alpha}\op(\frac{1}{D})[\sigma^{B,b}_{a}(\cdot,\xi)](x,\xi),
  \end{align*}
where,
 $$\fr_{x}(\op(\frac{1}{D}[\sigma^{B,b}_{a}(\cdot,\xi)])(\eta)=\frac{1}{i\eta}\fr_{x}(\sigma^{B,b}_{a}(x,\xi))(\eta),$$ which is well defined as $P_0(D)\op(\frac{1}{D})[\sigma^{B,b}_{a}(\cdot,\xi)]=0$ where $P_0(D)$ is the Littelwood-Paley projector defined in Section \ref{paracomposition_section Notations and functional analysis}. We then compute:
\[
L\circ  E_{approx}=\sigma^{B,b}_{\cdot}(Id-r), \text{ where: }
\]
\begin{align*}
   r:&C^{\beta+1-\alpha}\Gamma^m_0(\d \times \hat{\d} \setminus \b(0,b)) \rightarrow  \psi^{B,b}\bigg(C^k\Gamma^{\beta+1-\alpha}_0(\d \times \hat{\d} \setminus \b(0,b))\bigg) \\
  &a \mapsto \frac{\alpha (\alpha-1)}{4\pi}\int_0^1\int_{\d\times \hat{\d}}e^{i(x-y)\eta}\sigma^{B,b}_{\xi\abs{\xi}^{\alpha-2}}(\xi+t\eta)\psi^{B,b}(\eta,\xi)\sigma^{B,b}_{\frac{1}{\alpha}\partial_x a\abs{\xi}^{1-\alpha}}(y,\xi)dyd\eta dt.
\end{align*}
Let us remark that the remainder can also be written as:
\[
r(a)(x,\xi)=\frac{\alpha-1}{2}\underbrace{\op_x\bigg(\int^1_0\sigma^{B,b}_{\xi\abs{\xi}^{\alpha-2}}(\xi+t\eta)\psi^{B,b}(\eta,\xi)dt\bigg)}_{(*)}[\sigma^{B,b}_{\partial_x a\abs{\xi}^{1-\alpha}}](x,\xi),
\]
where $(*)$ is seen a Fourier multiplier in the $x$ variable for $\xi$ fixed. Thus estimating the semi-norms of $r(a)$ using the continuity of Fourier multipliers combined with the Bernstein inequalities we get we get by the frequency localisation of Paradifferential operators:
\[
M^{\beta+1-\alpha}_0(r(a);0)\leq \frac{C}{B} M^{\beta+1-\alpha}_0(\sigma^{B,b}_{a};0).
\]
Thus for $B$ sufficiently large $L$ has a right inverse on $\psi^{B,b}\bigg(C^0\Gamma^{\beta+1-\alpha}_1(\r \setminus \b(0,b))\bigg)$ given by the Neumann series:
\[
E=\sum^{+\infty}_{k=0}E_{approx}r^k.
\]
Thus we have constructed a $p\in C^0\Gamma^{\beta+1-\alpha}_1(\r \setminus \b(0,b))$ such that:
\begin{equation*}
\sigma^{B,b}_{p}\otimes i\xi \abs{\xi}^{\alpha-1}-i\xi \abs{\xi}^{\alpha-1}\otimes\sigma^{B,b}_{p} =\sigma^{B,b}_{a}.
\end{equation*}
Now we want to prove that $p\in C^k\Gamma^{\beta+1-\alpha}_1(\r \setminus \b(0,b))$ for all k. We  start by the following computation that comes from commuting with $ix$:
\begin{equation}\label{glob CP DBeq_sec impl constr symb_th constr comm_proof_eq1}
\sigma^{B,b}_{\partial_\xi p}\otimes i\xi \abs{\xi}^{\alpha-1}-i\xi \abs{\xi}^{\alpha-1}\otimes\sigma^{B,b}_{\partial_\xi p} =\sigma^{B,b}_{\partial_\xi a}+\alpha[i \abs{\xi}^{\alpha-1}\otimes\sigma^{B,b}_{p}-\sigma^{B,b}_{p}\otimes  i\abs{\xi}^{\alpha-1}].
\end{equation}
Now to get the desired bound we use the following lemma.
\begin{lemma}\label{glob CP DBeq_sec impl constr symb_th constr comm_proof_lem est comm constr}
let $p\in \sr'(\d\times \hat{\d})$ be a symbol such that for some cut-off parameters $B,b$ we have:
\[
\sigma^{B,b}_{p}\otimes i\xi \abs{\xi}^{\alpha-1}-i\xi \abs{\xi}^{\alpha-1}\otimes\sigma^{B,b}_{ p}=\sigma^{B,b}_{ a},
\]
for some $a\in C^0\Gamma^{\beta}_0(\r \setminus \b(0,b))$ with $\alpha>0$ and $\beta \in \r$. 

Then $\partial_x \sigma^{B,b}_{p}\in C^0\Gamma^{\beta+1-\alpha}_0(\r \setminus \b(0,b))$ and moreover we have the estimate:
\[
M^{\beta+1-\alpha}_0(\partial_x \sigma^{B,b}_{p};0)\leq \frac{M^{\beta}_0(\sigma^{B,b}_{a};0)}{B[1-(1-\frac{1}{B})^\alpha]}.
\]
\end{lemma}
\begin{proof}[Proof of  Lemma \ref{glob CP DBeq_sec impl constr symb_th constr comm_proof_lem est comm constr}]
Without loss of generality, we suppose $p\in\sr$ as the result can be deduced  by a standard density argument. We rewrite the identity verified by $p$ as follows:
\[
\frac{1}{2\pi}\int_0^1\int_{\d\times \hat{\d}}e^{i(x-y)\eta}\sigma^{B,b}_{\abs{\xi}^{\alpha-1}}(\xi+t\eta)\psi^{B,b}(\eta,\xi)\sigma^{B,b}_{\partial_x p}(y,\xi)dyd\eta dt=\frac{1}{\alpha}\sigma^{B,b}_{ a},
\]
thus,
\[
\underbrace{\op_x\bigg(\int^1_0\sigma^{B,b}_{\abs{\xi}^{\alpha-1}}(\xi+t\eta)\psi^{B,b}(\eta,\xi)dt\bigg)}_{(*)}[\sigma^{B,b}_{\partial_x p}](x,\xi)=\frac{1}{\alpha}\sigma^{B,b}_{a},
\]
where $(*)$ is an elliptic Fourier multiplier in the $x$ variable for $\xi$ fixed with the bound:
\[
\frac{B}{\alpha}[1-(1-\frac{1}{B})^\alpha]\abs{\xi}^{\alpha-1}=\int_0^1\bigg(1-\frac{t}{B}\bigg)^{\alpha-1}dt\abs{\xi}^{\alpha-1}\leq \int^1_0\sigma^{B,b}_{\abs{\xi}^{\alpha-1}}(\xi+t\eta)\psi^{B,b}(\eta,\xi)dt
\]
which gives the desired result.
\end{proof}
Getting back to the proof of Theorem \ref{glob CP DBeq_sec impl constr symb_th constr comm} and applying Lemma \ref{glob CP DBeq_sec impl constr symb_th constr comm_proof_lem est comm constr} to \eqref{glob CP DBeq_sec impl constr symb_th constr comm_proof_eq1} we get $p\in C^1\Gamma^{\beta+1-\alpha}_1(\r \setminus \b(0,b))$ and iterating \eqref{glob CP DBeq_sec impl constr symb_th constr comm_proof_eq1} we get $p\in C^k\Gamma^{\beta+1-\alpha}_1(\r \setminus \b(0,b))$ for all $k$.
\end{proof}
A corollary of Lemma \ref{glob CP DBeq_sec impl constr symb_th constr comm_proof_lem est comm constr} is the following unicity result.
\begin{corollary}\label{glob CP DBeq_sec impl constr symb_cor lem est comm constr}
let $p\in \sr'(\d\times \hat{\d})$ be a symbol such that for some cut-off parameters $B,b$ we have:
\[
\sigma^{B,b}_{p}\otimes i\xi \abs{\xi}^{\alpha-1}-i\xi \abs{\xi}^{\alpha-1}\otimes\sigma^{B,b}_{ p}=0.
\]
Then $\sigma^{B,b}_{p}$ is a Fourier multiplier, that is there exists a Fourier multiplier $m$ such that: 
\[
\sigma^{B,b}_{p}(x,\xi)=\sigma^{B,b}_{m}(x,\xi)=\psi^{B,b}(0,\xi)m(\xi).
\]
\end{corollary}

The benefit of giving a robust fixed point proof to Theorem \ref{glob CP DBeq_sec impl constr symb_th constr comm} is that it generalizes simply to the following perturbations.
\begin{corollary}\label{glob CP DBeq_sec impl constr symb_th constr comm+conjugation}
Consider two real numbers $\alpha\geq 1$, $\beta\leq \alpha-1$ and two symbol $a\in \Gamma_0^\beta(\d)$ and $p\in \Gamma_1^{\beta+\alpha-1}(\d)$. Then there exist $B>1$ and $\eps>0$ such that for 
\[
M^{\beta+\alpha-1}_0(p)\leq \eps,
\]
there exists a symbol $h\in\Gamma_1^{\beta+1-\alpha}(\d)$ such that, 
\begin{equation}\label{glob CP DBeq_sec impl constr symb_th constr comm+conjugation_para eq}
\sigma^{B,b}_{\crch{\int_0^1\sigma^{B,b}_{e^{i(1-r)p}_\otimes}\otimes \sigma^{B,b}_h\otimes \sigma^{B,b}_{e^{irp}_\otimes}dr,\xi\abs{\xi}^{\alpha-1}}} =\sigma^{B,b}_{a}.
\end{equation}
\end{corollary}
\begin{proof}
We notice that the problem is linear in $h$ and that for $p=0$ we are exactly in the case of the previous Theorem. The fixed point scheme proof presented above thus generalizes in a straightforward manner to this case.
\end{proof}

The following statement shows that if the considered symbols depend on time in a tame way then a time dependent version of Theorem \ref{glob CP DBeq_sec impl constr symb_th constr comm} still holds.

\begin{theorem}\label{glob CP DBeq_sec impl constr symb_th constr comm and time deriv}
Consider two real numbers $\alpha\geq 1$, $\beta\in \r$, a symbol $a\in C([0,T],\Gamma_0^\beta(\d))$ and $B$ given by Theorem \ref{glob CP DBeq_sec impl constr symb_th constr comm}. 
Suppose that for $j\in \n$ and $k\in \set{0,1,2}$ we have the following estimate:
\begin{equation}\label{glob CP DBeq_sec impl constr symb_th constr comm and time deriv_hypotehsis growthh time deriv}
M_0^{\beta+j\alpha}(\partial^j_t \partial_\xi^k a;0)\leq C K^j 
\end{equation}
for a constant $K<\alpha$.
Then there exist a constant $B'\geq B$ and symbol $p$ with, for $j\in \n$, \[p\in C\prt{[0,T],C^{2} \Gamma_1^{\beta+1-\alpha}(\d)} \text{ and } \partial^{1+j}_t p \in C\prt{[0,T],C^{2} \Gamma_0^{\beta+j\alpha}(\d)} \] such that, 
\begin{equation}\label{glob CP DBeq_sec impl constr symb_th constr comm and time deriv_para eq}
-\sigma^{B',b}_{\partial_t p}+\sigma^{B',b}_{p}\otimes i\xi \abs{\xi}^{\alpha-1}-i\xi \abs{\xi}^{\alpha-1}\otimes\sigma^{B',b}_{p} =\sigma^{B',b}_{a}.
\end{equation}
\end{theorem}
\begin{proof}
We define iteratively,
\[
\sigma^{B',b}_{p_0}\otimes i\xi \abs{\xi}^{\alpha-1}-i\xi \abs{\xi}^{\alpha-1}\otimes\sigma^{B',b}_{p_0}=\sigma^{B',b}_{a},
\]
\[
\sigma^{B',b}_{p_{j+1}}\otimes i\xi \abs{\xi}^{\alpha-1}-i\xi \abs{\xi}^{\alpha-1}\otimes\sigma^{B',b}_{p_{j+1}}=\sigma^{B',b}_{\partial_t p_{j}},
\]
then by the estimate \eqref{glob CP DBeq_sec impl constr symb_th constr comm_est costr comm}, for $B'$ sufficiently large $K<B'\crch{1-(1-\frac{1}{B})^\alpha}$ and hypothesis \eqref{glob CP DBeq_sec impl constr symb_th constr comm and time deriv_hypotehsis growthh time deriv} we get
$$p=\sum^{+\infty}_{j=0}p_j \in C^{0}\Gamma_1^{\beta+1-\alpha}(\d) \ \text{ and } \  \partial^{1+j}_t p \in C\prt{[0,T],C^{0} \Gamma_0^{\beta+j\alpha}(\d)}.$$
Now to get
$$p=\sum^{+\infty}_{j=0}p_j \in C^1\Gamma_1^{\beta+1-\alpha}(\d) \ \text{ and } \  \partial^{1+j}_t p \in C\prt{[0,T],C^{1} \Gamma_0^{\beta+j\alpha}(\d)},$$
one is tempted to again apply \eqref{glob CP DBeq_sec impl constr symb_th constr comm_est costr comm derv xi} with hypothesis \eqref{glob CP DBeq_sec impl constr symb_th constr comm and time deriv_hypotehsis growthh time deriv}, which indeed works for $K<\frac{\alpha-1}{\alpha}$. To avoid this, we keep the notation of the proof of Theorem \ref{glob CP DBeq_sec impl constr symb_th constr comm} and write:
\[
p_{j+1}=E(p_j)=E^{j+1}(\partial^j_t \sigma^{B',b}_a)\sim \frac{\abs{\xi}^{(j+1)(1-\alpha)}}{\alpha^{j+1}}\partial_x^{-(j+1)}\partial_t^{j}a,
\]
deriving in $\xi$ we see that in the main term the constants created are of magnitude $(j+1)(\alpha-1)$ which can be compensated for by the geometric growth for $K<\alpha$ and treating the error given by the semi-norms of $r$ in the same manner as before by taking $B'$ sufficiently large. Thus we get $$p=\sum^{+\infty}_{j=0}p_j \in C^1\Gamma_1^{\beta+1-\alpha}(\d) \ \text{ and } \  \partial^{1+j}_t p \in C\prt{[0,T],C^{1} \Gamma_0^{\beta+j\alpha}(\d)},$$  
and by iteration 
$$p=\sum^{+\infty}_{j=0}p_j \in C^2\Gamma_1^{\beta+1-\alpha}(\d) \ \text{ and } \  \partial^{1+j}_t p \in C\prt{[0,T],C^{2} \Gamma_0^{\beta+j\alpha}(\d)}.$$  
\end{proof}
We note that the previous proof was a rudimentary application of a linear Nash-Moser Scheme. We also remark that by taking $B'$ larger and larger we could have taken $$p=\sum^{+\infty}_{j=0}p_j \in C^k\Gamma_1^{\beta+1-\alpha}(\d) \ \text{ and } \  \partial^{1+j}_t p \in C\prt{[0,T],C^{k} \Gamma_0^{\beta+j\alpha}(\d)},$$  
but the method of the proof above does not give a $B'$ uniformly in $k$.

Before we turn to the non linear versions of the previous theorem we give a slight generalisation of the previous theorem which will amount to inverting locally around $0$ the linear problem in Theorem \ref{glob CP DBeq_sec impl constr symb_thm constr gauge trsf with time deriv}.

\begin{corollary}\label{glob CP DBeq_sec impl constr symb_thm constr gauge trsf with time deriv_corollary linearilised pbm}
Consider three real numbers $\eps>0$, $\alpha\geq 1$, $\beta\leq \alpha -1$ and two symbol $a,b\in \Gamma_0^\beta(\d)$ and $p\in \Gamma_1^{\beta+\alpha-1}(\d)$. 
Suppose that for $j\in \n$ and $k\in \set{0,1,2}$ we have the following estimate:
\begin{equation}\label{glob CP DBeq_sec impl constr symb_th constr comm and time deriv_hypotehsis growthh time deriv}
M_0^{\beta+j\alpha}(\partial^j_t \partial_\xi^k a;0)\leq C K^j, \ M_0^{\beta+j\alpha}(\partial^j_t \partial_\xi^k b;0)\leq C \eps K^j \text{ and } M_0^{\beta+j\alpha}(\partial^j_t \partial_\xi^k p;0)\leq C \eps K^j
\end{equation}
for a constant $K<\alpha$.
Then for $\eps$ sufficiently small there exist a constant $B>1$ and symbol $h$ with, for $j\in \n$, \[h\in C\prt{[0,T],C^{2} \Gamma_1^{\beta+1-\alpha}(\d)} \text{ and } \partial^{1+j}_t h \in C\prt{[0,T],C^{2} \Gamma_0^{\beta+j\alpha}(\d)} \] such that, 
\begin{align}\label{glob CP DBeq_sec impl constr symb_th constr comm and time deriv_hypotehsis growthh time deriv_para eq}
&-\sigma^{B,b}_{\int_0^1\sigma^{B,b}_{e^{i(1-r)p}_\otimes}\otimes \sigma^{B,b}_{i\partial_th}\otimes \sigma^{B,b}_{e^{irp}_\otimes}dr}+\sigma^{B,b}_{\crch{\int_0^1\sigma^{B,b}_{e^{i(1-r)p}_\otimes}\otimes \sigma^{B,b}_h\otimes \sigma^{B,b}_{e^{irp}_\otimes}dr,\xi\abs{\xi}^{\alpha-1}}} =\sigma^{B,b}_{a}.\\
&-\sigma^{B,b}_{\int_0^1\sigma^{B,b}_{e^{i(1-r)p}_\otimes}\otimes \sigma^{B,b}_h \otimes \sigma^{B,b}_{e^{irp}_\otimes} \otimes \sigma^{B,b}_{b}}
+\sigma^{B,b}_{\int_0^1\int_0^1 \sigma^{B,b}_{e^{i(1-r)p}_\otimes}\otimes \sigma^{B,b}_{\partial_tp}\otimes \sigma^{B,b}_{e^{i(r-u)p}_\otimes}\otimes \sigma^{B,b}_{h}\otimes \sigma^{B,b}_{e^{iup}_\otimes}drdu}\nonumber \\
&+\sigma^{B,b}_{\int_0^1\int_0^r\crch{\sigma^{B,b}_{e^{i(u-r)p}_\otimes}\otimes \otimes \sigma^{B,b}_{\partial_t p} \sigma^{B,b}_{e^{i(r-u)p}_\otimes},\sigma^{B,b}_{e^{-irp}_\otimes}\otimes \otimes \sigma^{B,b}_h \sigma^{B,b}_{e^{irp}_\otimes}}dr du}\nonumber
\end{align}
\end{corollary}
\begin{proof}
We notice that the problem is linear in $h$ and that for $p=0$ and $b=0$ we are exactly in the case of the previous Theorem. The fixed point scheme and linear Nash-Moser scheme of the previous proofs are stable under the hypothesis on $p$ and $b$.
\end{proof}

We now give a non linear version of Theorem \ref{glob CP DBeq_sec impl constr symb_th constr comm} that emphasizes the role of the paradifferential setting in solving those implicit function problems in the non linear case.
\begin{theorem}\label{glob CP DBeq_sec impl constr symb_thm constr gauge trsf}
Consider two real numbers $\alpha\geq 1$, $\beta\leq \alpha-1$ and a symbol $a\in \Gamma_0^\beta(\d)$ Then there exists $\eps>0, B>1$ such that for 
\[
M_0^\beta(a;0)\leq \eps,
\]
there exists a symbol $p\in  \Gamma_1^{\beta+1-\alpha}(\d)$ such that 
\begin{equation}\label{glob CP DBeq_sec impl constr symb_thm constr gauge trsf_eq1}
\sigma^{B,b}_{e^{i p}_\otimes }\otimes i\xi \abs{\xi}^{\alpha-1}-i\xi \abs{\xi}^{\alpha-1}\otimes \sigma^{B,b}_{ e^{i p}_\otimes}=\sigma^{B,b}_{ a}.
\end{equation}
Moreover we have the estimates:
\begin{equation}\label{glob CP DBeq_sec impl constr symb_thm constr gauge trsf_est smnrm}
M_0^{\beta+1-\alpha}(\partial_x \sigma^{B,b}_{p};0) \leq \frac{1+C\eps}{\alpha} M_0^\beta(\sigma^{B,b}_{a};0),
\end{equation}
\begin{align}\label{glob CP DBeq_sec impl constr symb_thm constr gauge trsf_est smnrm drv}
M_0^{\beta-\alpha}(\partial_\xi \partial_x \sigma^{B,b}_{p};0) &\leq (1+CM_0^\beta(\partial_\xi \sigma^{B,b}_{a};0))\times \bigg[ \frac{M^{\beta-1}_0(\partial_\xi\sigma^{B,b}_{a};0)}{B[1-(1-\frac{1}{B})^\alpha]}\\
&+\alpha\frac{(1+\frac{1}{B})^{\alpha-1}-1}{1-(1-\frac{1}{B})^\alpha}M^{\beta+1-\alpha}_0(\partial_x \sigma^{B,b}_{p};0)\bigg].\nonumber
\end{align}
\end{theorem}

We note that $\beta+1-\alpha\leq 0$ thus the hypothesis on $\IM(p)$ is automatically verified. 
\begin{proof}
The proof should in spirit amount to a Nash-Moser scheme as we are looking to prove an implicit function type of result in Fr\'echet. In our case the problem is simpler due to the following key observation, for $k\geq 0$ the underlining map is well defined on the Banach spaces in the scale defining the Fr\'echet space of paradifferential operators:
\begin{align*}
   F:&C^k\Gamma^{\beta+1-\alpha}_1(\d \times \hat{\d} \setminus \b(0,b)) \rightarrow  \psi^{B,b}\bigg(C^k\Gamma^\beta_0(\d \times \hat{\d} \setminus \b(0,b))\bigg) \\
  &p\mapsto \sigma^{B,b}_{ e^{i p}_\otimes}\otimes i\xi \abs{\xi}^{\alpha-1}-i\xi \abs{\xi}^{\alpha-1}\otimes \sigma^{B,b}_{ e^{i p}_\otimes}
  \end{align*}

The proof will again proceed in two steps, we first prove that $F$ has a right inverse on the Banach space $C^0\Gamma^{\beta+1-\alpha}_1(\r \setminus \b(0,b))$ and then we prove propagation of regularity in the frequency variable $\xi$ for solutions of the equation \eqref{glob CP DBeq_sec impl constr symb_th constr comm_para eq}.

Noticing that $F(0)=0$, the goal is thus to prove the local surjectivity of $F$ around the origin. Now that we reduced the problem to Banach spaces, by the inverse function theorem it suffices to find a right inverse to the differential of $F$ at $0$. Computing the differential at $0$ we get:
\[
D_0F(h)=\sigma^{B,b}_{h}\otimes i\xi \abs{\xi}^{\alpha-1}-i\xi \abs{\xi}^{\alpha-1}\otimes\sigma^{B,b}_{h}=L(h),
\]
Thus by Theorem \ref{glob CP DBeq_sec impl constr symb_th constr comm} and the local inversion theorem in Banach spaces we get the desired local surjectivity.

Now we turn to propagation of regularity in the $\xi$ variable, we fix:
$$p \in C^0\Gamma^{\beta+1-\alpha}_1(\d \times \hat{\d} \setminus \b(0,b)) \text{ and } a \in \Gamma^{\beta}_0(\d \times \hat{\d} \setminus \b(0,b)).$$ 
To make all of the computations rigorous we suppose $p \in \sr$ and the desired result is obtained by a density argument. The computation behind the propagation of regularity is the following analogue of \eqref{glob CP DBeq_sec impl constr symb_th constr comm_proof_eq1}. We start from:
\[
\sigma^{B,b}_{ e^{i p}_\otimes}\otimes i\xi \abs{\xi}^{\alpha-1}-i\xi \abs{\xi}^{\alpha-1}\otimes \sigma^{B,b}_{ e^{i p}_\otimes}=\sigma^{B,b}_{a},
\]
which from the analysis of the linear case gives 
\[
\partial_x \sigma^{B,b}_{\otimes e^{i p}}\in \Gamma^{\beta+1-\alpha}_0(\r \setminus \b(0,b)),
\]
we compute
\[
\partial_x \sigma^{B,b}_{ e^{i p}_\otimes}=\sigma^{B,b}_{\int_0^1  e^{ir p}_\otimes \otimes  \sigma^{B,b}_{\partial_x p}\otimes  e^{-ir p}_\otimes dr}
\in \Gamma^{\beta+1-\alpha}_0(\r \setminus \b(0,b))
\]

Getting back to the definition of $e^{ir p}_\otimes$ as $\beta+1-\alpha\leq 0$:
\[
\sigma^{B,b}_{ e^{ir p}_\otimes}=\sum^{\infty}_{k=0}\frac{i^k r^k}{k!}\sigma^{B,b}_{\otimes^k \sigma^{B,b}_{p}},
\]
thus,
\[
\sigma^{B,b}_{ e^{ir p}_\otimes}=1+O_{M^0_0(\cdot;0)}(\eps),
\]
which gives:
\[
\sigma^{B,b}_{\partial_x p}
\in \Gamma^{\beta+1-\alpha}_0(\r \setminus \b(0,b)).\]
\end{proof}

Finally we give the Theorem that permits the construction of the gauge transform for $\alpha>2$.
\begin{theorem}\label{glob CP DBeq_sec impl constr symb_thm constr gauge trsf with time deriv}
Consider two real numbers $\alpha\geq 1$, $\beta\leq \alpha-1$ and a symbol $a\in C([0,T],\Gamma_0^\beta(\d))$. Then there exists $\eps>0, B>1$ such that for $j\in \n$ and $k\in \set{0,1,2}$,
\begin{equation}\label{glob CP DBeq_sec impl constr symb_thm constr gauge trsf_hypotehsis growthh time deriv}
M_0^{\beta+k\alpha}(\partial^j_t \partial^k_\xi a;0)\leq \eps K^j,
\end{equation}
with $K<\alpha$. Then there exist a symbol p with, for $j\in \n$, \[p\in C\prt{[0,T],C^{2} \Gamma_1^{\beta+1-\alpha}(\d)} \text{ and } \partial^{1+j}_t p \in C\prt{[0,T],C^{2} \Gamma_0^{\beta+j\alpha}(\d)} \] such that, 
\begin{equation}\label{glob CP DBeq_sec impl constr symb_thm constr gauge trsf_eq1}
-\sigma^{B,b}_{ \partial_t\crch{\sigma^{B,b}_{ e^{i p}_\otimes}}}+i\xi \abs{\xi}^{\alpha-1}\otimes \sigma^{B,b}_{ e^{i p}_\otimes}-\sigma^{B,b}_{e^{i p}_\otimes}\otimes i\xi \abs{\xi}^{\alpha-1}=\sigma^{B,b}_{\sigma^{B,b}_{ e^{i p}_\otimes} \otimes a}.
\end{equation}
\end{theorem}
\begin{proof}
Again the para-linear setting and the previous implicit function constructions will reduce the problem to the Banach space case. We start by defining the map:
\[
\Phi(p,a)=-\sigma^{B,b}_{ \partial_t\crch{\sigma^{B,b}_{ e^{i p}_\otimes}}}+i\xi \abs{\xi}^{\alpha-1}\otimes \sigma^{B,b}_{ e^{i p}_\otimes}-\sigma^{B,b}_{e^{i p}_\otimes}\otimes i\xi \abs{\xi}^{\alpha-1}-\sigma^{B,b}_{\sigma^{B,b}_{ e^{i p}_\otimes} \otimes a},
\]
\[
\Phi:C\prt{[0,T],C^{2} \Gamma_1^{\beta+1-\alpha}(\d)}\cap \dot{C}_1\prt{[0,T],C^{2} \Gamma_0^{\beta}(\d)}\rightarrow C\prt{[0,T],C^{2} \Gamma_1^{\beta+1-\alpha}(\d)},
\]
and the goal is to show that close to $\Phi(0,0)=0$ we have a solution to $\Phi(p,a)=0$ with \[p\in C\prt{[0,T],C^{2} \Gamma_1^{\beta+1-\alpha}(\d)} \text{ and } \partial^{1+j}_t p \in C\prt{[0,T],C^{2} \Gamma_0^{\beta+j\alpha}(\d)} \] for $a$ verifying the hypothesis of the theorem for some constants $B$ and $\eps$. We notice that equation \eqref{glob CP DBeq_sec impl constr symb_th constr comm and time deriv_hypotehsis growthh time deriv_para eq} of Corollary \ref{glob CP DBeq_sec impl constr symb_thm constr gauge trsf with time deriv_corollary linearilised pbm} can be reinterpreted as:
\[
D_{p}\Phi(p,b)[h]=a,
\]
and that Corollary \ref{glob CP DBeq_sec impl constr symb_thm constr gauge trsf with time deriv_corollary linearilised pbm} gives the right inverse to this linear map for $(b,p)$ sufficiently close to $0$ verifying the same growth hypothesis given here. Thus by the implicit function theorem in Banach spaces for $\eps>0$ and $B>1$ sufficiently small there exists 
\[
p\in C\prt{[0,T],C^{2} \Gamma_1^{\beta+1-\alpha}(\d)}\cap \dot{C}_1\prt{[0,T],C^{2} \Gamma_0^{\beta}(\d)}, \text{ solving } \Phi(p,a)=0
\]
that is 
\[
-\sigma^{B,b}_{ \partial_t\crch{\sigma^{B,b}_{ e^{i p}_\otimes}}}+i\xi \abs{\xi}^{\alpha-1}\otimes \sigma^{B,b}_{ e^{i p}_\otimes}-\sigma^{B,b}_{e^{i p}_\otimes}\otimes i\xi \abs{\xi}^{\alpha-1}=\sigma^{B,b}_{\sigma^{B,b}_{ e^{i p}_\otimes} \otimes a},
\]
which by iteration give the desired regularity on $p$.
\end{proof}

\section{Sobolev estimate on the weakly dispersive Burgers equation}
The goal of this section is to prove Theorem \ref{glob CP DBeq_intro_Para th CP imprv enrgy est}. First it suffice to make the estimate \eqref{glob CP DBeq_intro_Para th CP imprv enrgy est_energy est} for $u_0 \in C^\infty_0$ and deduce the general result by density. We take $u$ a solution to the Cauchy problem \eqref{glob CP DBeq_intro_Para th CP imprv enrgy est_ eq BDP para avec cutoff}.

By Theorem \ref{glob CP DBeq_sec impl constr symb_th constr comm} there exists $p \in \Gamma^{2-\alpha}_1(\r)$ such that:
\begin{equation}\label{glob CP DBeq_sec Sob est wk DB_eq1}
[T^{B,b}_{ip},T^{B,b}_{i\xi \abs{\xi}^{\alpha-1}}]=-iT^{B,b}_{\sigma^{B,b}_{u\xi}+(\sigma^{B,b}_{u\xi})^*}.
\end{equation}
Now by Corollary \ref{glob CP DBeq_sec impl constr symb_cor lem est comm constr} as $(T^{B,b}_{\xi \abs{\xi}^{\alpha-1}})^*=T^{B,b}_{\xi \abs{\xi}^{\alpha-1}}$ and the left hand side being $L^2$ skew-adjoint we get:
\begin{equation}\label{glob CP DBeq_sec Sob est wk DB_eq2}
\bigg(T^{B,b}_{p}\bigg)^*=T^{B,b}_{p} \text{ in $L^2$}.
\end{equation}

Following \cite{Ifrim17}, before the gauge transform step we start by a normal form one to put the equation in advantageous form that can immediately be handled by the gauge transform. For this we keep the notation of \cite{Molinet18} and recall Theorem $2.2$ on a generalized version of Coifman-Meyer's multipliers.
\begin{proposition}[Theorem $2.2$ of \cite{Molinet18}]\label{glob CP DBeq_sec Sob est wk DB_CM multipliers}
Let $1<p_1,\cdots,p_n<\infty$ and $1\leq p<+\infty$ satisfy $\frac{1}{p}=\frac{1}{p_1}+\cdots+\frac{1}{p_n}$. Assume that $f_1,\cdots,f_n\in \sr(\r)$ are functions with Fourier variables supported in $\{\abs{\xi}\sim N_i\}$ for some dyadic numbers $N_1,\cdots,N_n$.

For $n\geq 1$ and $\chi$ a bounded measurable function on $\r^n$ define the multilinear Fourier multiplier operator $\prod^n_{\chi}$ on $\sr(\r^n)$ by
\begin{equation}\label{glob CP DBeq_sec Sob est wk DB_multilinear multiplier}
\Pi^n_{\chi}(f_1,\cdots,f_n)(x)=\frac{1}{(2\pi)^{n}}\int_{\r^n}\chi(\xi_1,\cdots,\xi_n)\prod^n_{i=1}e^{ix(\xi_1+\cdots\xi_n)}\fr(f_i)(\xi_i)d\xi_1\cdots d\xi_n.
\end{equation}

Assume also that $\chi \in C^{\infty}(\r^n)$ satisfies the Marcinkiewicz type condition
\begin{equation}\label{glob CP DBeq_sec Sob est wk DB_Marcinkiewicz type condition}
\forall{\beta}=(\beta_1,\cdots,\beta_n)\in \n^n, \ \abs{\partial^\beta \chi(\xi)}\lesssim \prod^n_{i=1}\abs{\xi_i}^{\beta_i},
\end{equation}
on the support of $\prod^n_{i=1}\fr(f)(\xi)$. Then,
\begin{equation}\label{glob CP DBeq_sec Sob est wk DB_Lp multiplier estimate}
\forall{\beta}=(\beta_1,\cdots,\beta_n)\in \n^n, \ \norm{\Pi^n_{\chi}(f_1,\cdots,f_n)}_{L^p}\lesssim \prod^n_{i=1}\norm{f_i}_{L^{p_i}},
\end{equation}
with an implicit constant that doesn't depend on $N_1,\cdots,N_n$. 
\end{proposition}
We now get back to the equation:
\[
\partial_t u +T^{B,b}_{u}\partial_x u+\partial_x\abs{D}^{\alpha-1}u=0,
\]
and write for $v=\D^s u$,
\begin{multline*}
\partial_t v+i\frac{T^{B,b}_{ \D^{-s}v\xi}+\left(T^{B,b}_{ \D^{-s}v\xi}\right)^*}{2} v+\partial_x \abs{D}^{\alpha-1}v\\=[T^{B,b}_{\D^{-s}v}\partial_x,\D^s]\D^{-s}v-i\frac{T^{B,b}_{ \D^{-s}v\xi}-\left(T^{B,b}_{ \D^{-s}v\xi}\right)^*}{2}v.
\end{multline*}
We are looking for $w=v+B(v,v)$ to eliminate the worst terms in the write hand side, the equation on $w$ reads
\begin{align*}
    &\partial_t w+i\frac{T^{B,b}_{ \D^{-s}v\xi}+\left(T^{B,b}_{ \D^{-s}v\xi}\right)^*}{2} v+\partial_x \abs{D}^{\alpha-1}w\\
    &=[T^{B,b}_{\D^{-s}v}\partial_x,\D^s]\D^{-s}v-i\frac{T^{B,b}_{ \D^{-s}v\xi}-\left(T^{B,b}_{ \D^{-s}v\xi}\right)^*}{2}v\\
    &+D^{\alpha-1} B(\partial_xu,u)+D^{\alpha-1} B(u,\partial_x u)-B(D^{\alpha-1}\partial_x u,u)-B(u,D^{\alpha-1}\partial_x u)\\
    &-B(\D^sT_u\partial_x u, v)-B(v,\D^sT_u\partial_x u).
\end{align*}
Thus we are looking for be $B$ such that
\begin{multline*}
D^{\alpha-1} B(\partial_xu,u)+D^{\alpha-1} B(u,\partial_x u)-B(D^{\alpha-1}\partial_x u,u)-B(u,D^{\alpha-1}\partial_x u)\\
=[T^{B,b}_{\D^{-s}v}\partial_x,\D^s]\D^{-s}v-i\frac{T^{B,b}_{ \D^{-s}v\xi}-\left(T^{B,b}_{ \D^{-s}v\xi}\right)^*}{2}v.
\end{multline*}
Using the multiplier notation we get, $B=\Pi_{\chi}$ with
\begin{align*}
\chi(\xi_1,\xi_2)&=\frac{\psi^{B,b}(\xi_1,\xi_2)\langle \xi_2 \rangle^{-s}\langle \xi_1 \rangle^{-s}\xi_2\prt{\langle \xi_2 \rangle^{s}-\langle \xi_2+\xi_1 \rangle^{s}}}{\Omega_{\alpha}(\xi_1,\xi_2)}\\
&\times\frac{\psi^{B,b}(\xi_1,\xi_2-\xi_1)(\xi_2-\xi_1)-\psi^{B,b}(\xi_1,\xi_2)\xi_2}{2\Omega_{\alpha}(\xi_1,\xi_2)},
\end{align*}
which we rewrite as
\[
w=v+\Pi_{\chi_1}(u,D^{1-\alpha}v).
\]
where $\chi_1$ verifies the hypothesis of proposition \ref{glob CP DBeq_sec Sob est wk DB_CM multipliers}.

Injected back in the equation on $w$ we get 
\begin{align*}
    &\partial_t w+i\frac{T^{B,b}_{ \D^{-s}v\xi}+\left(T^{B,b}_{ \D^{-s}v\xi}\right)^*}{2} w+\partial_x \abs{D}^{\alpha-1}w\\
    &=-i\frac{T^{B,b}_{ \D^{-s}v\xi}+\left(T^{B,b}_{ \D^{-s}v\xi}\right)^*}{2}B(v,v)\\
    &-B(\D^sT_u\partial_x u, v)-B(v,\D^sT_u\partial_x u).
\end{align*}
By the explicit formula on $B$ on can re-arrange the terms on the right hand side to get:
\[
\partial_t w+[iT^{B,b}_p,\partial_x \abs{D}^{\alpha-1}]w+\partial_x \abs{D}^{\alpha-1}w=\Pi_{\chi_{2}}(T_u\partial_x u,D^{1-\alpha}v),
\]
where $\chi_2$ still verifies the hypothesis of proposition \ref{glob CP DBeq_sec Sob est wk DB_CM multipliers} and we used the definition of $p$. 
Conjugating with $e^{iT^{B,b}_{p}}$ we get
\begin{multline*}
\partial_t w+e^{-iT^{B,b}_{p}}\partial_x \abs{D}^{\alpha-1}e^{iT^{B,b}_{p}}w-\frac{1}{2}\int_0^1e^{-irT^{B,b}_{p}}\crch{T^{B,b}_{p},T^{B,b}_{\sigma^{B,b}_{u\xi}+(\sigma^{B,b}_{u\xi})^*}}e^{irT^{B,b}_{p}}\\=e^{iT^{B,b}_{p}}\Pi_{\chi_{2}}(T_u\partial_x u,D^{1-\alpha}v),
\end{multline*}
The main observation here is that due to that the implicit construction of symbols we were able to keep the following cancellation, even though $$\crch{T^{B,b}_{p},T^{B,b}_{\sigma^{B,b}_{u\xi}+(\sigma^{B,b}_{u\xi})^*}} \in \Gamma_0^{3-\alpha}$$ we have that 
\begin{multline*}
\prt{\int_0^1e^{-irT^{B,b}_{p}}\crch{T^{B,b}_{p},T^{B,b}_{\sigma^{B,b}_{u\xi}+(\sigma^{B,b}_{u\xi})^*}}e^{irT^{B,b}_{p}}}^*\\=-\int_0^1e^{-irT^{B,b}_{p}}\crch{T^{B,b}_{p},T^{B,b}_{\sigma^{B,b}_{u\xi}+(\sigma^{B,b}_{u\xi})^*}}e^{irT^{B,b}_{p}}.
\end{multline*}
Making an energy estimate in $w$ we get by integration by parts
\[
\frac{d}{dt}\norm{w}_{L^2}\leq C\norm{\partial_x D^{1-\alpha}(u^2)}_{L^\infty_x}\norm{v}_{L^2}
\]
From the normal form transform
\[
w=v+\Pi_{\chi_1}(u,D^{1-\alpha}v).
\]
we see that for $\norm{u}_{H^{\prt{\frac{3}{2}-\alpha}^+}}$ sufficiently small we have that 
\[
C^{-1}\norm{v}_{L^2} \leq \norm{w}_{L^2}\leq C\norm{v}_{L^2},
\]
which injected back in the energy estimate gives the desired bound.
\section{Complete gauge transform for the dispersive Burgers equation}
In contrast to the previous section in the first step to conjugating \eqref{glob CP DBeq_intro_Para th conjg_eq1} we make the choice, by Theorem \ref{glob CP DBeq_sec impl constr symb_thm constr gauge trsf with time deriv}, of a symbol p with, for $j\in \n$, \[p\in C\prt{[0,T],C^{2} \Gamma_1^{0}(\d)} \text{ and } \partial^{1+j}_t p \in C\prt{[0,T],C^{2} \Gamma_0^{(j+1)\alpha-1}(\d)} \] such that:
\begin{equation}\label{glob CP DBeq_sec cmpl gg trsm_eq6}
-\partial_t\sigma^{B,b}_{e^{i p}_\otimes}+\sigma^{B,b}_{e^{i p}_\otimes}\otimes i\xi \abs{\xi}^{\alpha-1}-i\xi \abs{\xi}^{\alpha-1}\otimes \sigma^{B,b}_{ e^{i p}_\otimes}=-\sigma^{B,b}_{\sigma^{B,b}_{ e^{i p}_\otimes}\otimes iu\xi}.
\end{equation}
Indeed in order have this we need to check the tameness of the dependence of $\sigma^{B,b}_{iu\xi}$ in time, which from the equation on $u$ and the frequency cut-off bounding time derivatives with $\alpha$ space derivatives we get for $j\in \n$
\[
M_0^{(j+1)\alpha-1}(\partial^{j}_t (\sigma^{B,b}_{iu\xi});0)\leq C \norm{u}_{C_*^{(2-\alpha)^+}}(2^j+1)
\]
which insures the existence of $p$ for $\alpha>2$ and $\norm{u}_{C_*^{(2-\alpha)^+}}$ sufficiently small by Theorem \ref{glob CP DBeq_sec impl constr symb_thm constr gauge trsf with time deriv}.

For the proof of well-posedness we also need to see that for 2 different solutions $u,v$ we have by the proof of Theorem \ref{glob CP DBeq_sec impl constr symb_thm constr gauge trsf with time deriv}:
\[ M_0^{0}([p(u)-p(v)](t,\cdot);1)\leq e^{C\norm{(u,v)}_{C^{(2-\alpha)^+}_*}}\norm{\op(\frac{1}{D})P_{\geq b}[u-v](t,\cdot)}_{C^{(2-\alpha)^+}_*}.\]

Now looking at the equation on $u$:
\[
\partial_t u+T_u^{B,b}\partial_x u+\partial_x\abs{D}^{\alpha-1}u=0,
\]
put $w=T^{B,b}_{e^{iT^{B,b}_p}}u$ then by construction of $p$ we get:
\[
\partial_t w +\partial_x\abs{D}^{\alpha-1}w=\prt{T^{B,b}_{\sigma^{B,b}_{ e^{i p}_\otimes}\otimes iu\xi}-T^{B,b}_{ e^{i p}_\otimes}T_u^{B,b}\partial_x}u=T^{\frac{B^2}{2B+1}}_{R(u)}w.
\]
As $p$ is $C^2$ in $\xi$ the residual term $T^{\frac{B^2}{2B+1}}_{R(u)}$ is a an operator of order $0$ in $H^s$ for all $s$ for $u\in C_*^{(2-\alpha)^+}$ by \eqref{paracomposition_Notions of microlocal analysis_Paradifferential Calculus_difference between 2 choices of cutoff} and Theorem \ref{Cont of lim reg exotic sym_thm for para symb}. Thus an energy estimate on $w$ gives 
\[
\norm{w(t)}_{H^s}\leq e^{C\norm{u}_{L^\infty_tC_*^{(2-\alpha)^+}}}\norm{w_0}.
\]
To get back to $u$ by the ellipticity of $T^{B,b}_{e^{iT^{B,b}_p}}$ we see that:
\[
\norm{u(t)}_{H^s}\leq e^{C\norm{u}_{L^\infty_tC_*^{(2-\alpha)^+}}}\prt{\norm{w(t)}_{H^s}+\norm{P_0(D)u}_{L^2}}.
\]
To treat the low frequency component we notice that for the para-linear equation by projecting on $P_0(D)$ we get:
\[
\partial_t P_0(D)u+\partial_x\abs{D}^{\alpha-1}P_0(D)u=0,
\]
thus, injecting back into the energy estimate, we get the apriori estimate
\[
\norm{u(t)}_{H^s}\leq e^{C\norm{u}_{L^\infty_tC_*^{(2-\alpha)^+}}}\norm{u_0}_{H^s}.
\]
Now for the Lipschitz estimate, taking another solution $v$ and defining analogously $p'$ and $w'$ we get 
\[
\partial_t w' +\partial_x\abs{D}^{\alpha-1}w'=\prt{T^{B,b}_{\sigma^{B,b}_{ e^{i p'}_\otimes}\otimes iv\xi}-T^{B,b}_{ e^{i p'}_\otimes}T_v^{B,b}\partial_x}v=T^{\frac{B^2}{2B+1}}_{R(v)}w',
\]
thus,
\[
\partial_t (w'-w) +\partial_x\abs{D}^{\alpha-1}(w'-w)=T^{\frac{B^2}{2B+1}}_{R(v)-R(u)}w'+T^{\frac{B^2}{2B+1}}_{R(u)}(w'-w),
\]
Making an energy estimate we get
\[
\norm{w'-w}_{H^s}\leq e^{C\norm{u}_{L^\infty_tC_*^{(2-\alpha)^+}}}\crch{\norm{w'_0-w_0}_{H_s}+\norm{u-v}_{C_*^{(2-\alpha)^+}}\norm{w'}_{H^s}}.
\]
Now writing 
\begin{align*}
v-u&=T^{B,b}_{e^{-iT^{B,b}_{p'}}}w'-T^{B,b}_{e^{-iT^{B,b}_p}}w +\prt{Id-T^{B,b}_{e^{-iT^{B,b}_{p'}}}T^{B,b}_{e^{iT^{B,b}_{p'}}}}v\\
&-\prt{Id-T^{B,b}_{e^{-iT^{B,b}_p}}T^{B,b}_{e^{iT^{B,b}_p}}}u\\
&=T^{B,b}_{e^{-iT^{B,b}_{p'}}}(w'-w)+\crch{T^{B,b}_{e^{-iT^{B,b}_{p'}}}-T^{B,b}_{e^{-iT^{B,b}_p}}}w\\
&+\prt{Id-T^{B,b}_{e^{-iT^{B,b}_{p'}}}T^{B,b}_{e^{iT^{B,b}_{p'}}}}(v-u)-\prt{T^{B,b}_{e^{-iT^{B,b}_{p'}}}T^{B,b}_{e^{iT^{B,b}_{p'}}}-T^{B,b}_{e^{-iT^{B,b}_p}}T^{B,b}_{e^{iT^{B,b}_p}}}u.
\end{align*}
Combining the previous identity with the equations on the low frequencies and the energy estimate on $w-w'$ we get
\[
\norm{v-u}_{H^s}\leq e^{C\norm{(u,v}_{L^\infty_tC_*^{(2-\alpha)^+}}}\norm{v-u}_{H_s}.
\]
The Lipschitz bound combined with the $H^s$ apriori bound ensure the well posendess of the Cauchy problem for $s>\frac{1}{2}+2-\alpha$.
\appendix

\section{Paradifferential Calculus}\label{paracomposition_Notions of microlocal analysis_Paradifferential Calculus}
In this paragraph we review classic notations and results about paradifferential and pseudodifferential calculus that we need in this paper.
We follow the presentations in \cite{Hormander71}, \cite{Hormander97}, \cite{Taylor07}, and \cite{Metivier08} which give an accessible and complete presentation. 
\begin{notation}
 In the following presentation we will use the usual definitions and standard notations for the regular functions $C^k$, $C^k_b$ for bounded ones and $C^k_0$ for those with compact support,
  the distribution space $\dr'$,$\er'$ for those with compact support, 
  $\dr'^k$,$\er'^k$ for distributions of order k, Lebesgue spaces ($L^p$), Sobolev spaces ($H^s,W^{p,q}$) 
  and the Schwartz class $\sr$ and it's dual $\sr'$. All of those spaces are equipped with their standard topologies. We also use the \textit{Landau notation}  $O_{\norm{ \ }}(X)$.
\end{notation}

For the definition of the periodic symbol classes we will need the following definitions and notations.
\begin{notation}
We will use $\d$ to denote $\t$ or $\r$ and $\hat{\d}$ to denote their duals that is $\z$ in the case of $\t$ and $\r$ in the case of $\r$. For concision an integral on $\z$ that is $\displaystyle \int_\z$ should be understood as $\displaystyle \sum_\z$. A function $a$ is said to be in $ C^\infty(\t \times \z)$ if for every $\xi \in \z$, $a(\cdot,\xi) \in C^\infty(\t)$. For $\xi\in \z$, $\partial_\xi$ should be understood as the forward difference operator, that is
\[\partial_\xi a(\xi)=a(\xi+1)-a(\xi),\ \xi \in \z.\]
We recall the following simple identities for the Fourier transform on the torus:
\[
\begin{cases}
\fr_{\t}(\partial_x^\alpha f)(\xi)=\xi^\alpha \fr_{\t}(f)(\xi), \xi \in \z,\\
\fr_{\t}((e^{-2i\pi x}-1)^\alpha f)(\xi)=\xi^\alpha \fr_{\t}(f)(\xi), \xi \in \z.
\end{cases}
\]
 \end{notation}

\subsection{Littlewood-Paley Theory}\label{paracomposition_section Notations and functional analysis}
\begin{definition}[Littlewood-Paley decomposition]\label{paracomposition_section Notations and functional analysis_def LP Theory}
Pick $P_0\in C^\infty_0(\r)$ so that: $$P_0(\xi)=1 \text{ for } \abs{\xi}<1 \text{ and } 0 \text{ for }  \abs{\xi}>2 .$$ 
We define a dyadic decomposition of unity by:
\[ \text{for } k \geq 1, \ P_{\leq k}(\xi)=P_0(2^{-k}\xi), \ P_k(\xi)=P_{\leq k}(\xi)-P_{\leq k-1}(\xi). \]
 Thus,\[ P_{\leq k}(\xi)=\sum_{0\leq j \leq k}P_j(\xi) \text{ and } 1=\sum_{j=0}^\infty P_j(\xi). \]
 Introduce the operator acting on $\mathscr S '(\r)$: 
 \[P_{\leq k}u=\fr^{-1}(P_{\leq k}(\xi)u) \text{ and } u_k=\fr^{-1}(P_k(\xi)u).\]
 Thus,
 \[u=\sum_k u_k.\]
 Finally put $\set{k\geq 1, C_k=\supp \ P_k}$ the set of rings associated to this decomposition.
\end{definition}

\begin{remark}
An interesting property of the Littlewood-Paley decomposition is that even if the decomposed function is merely a distribution the terms of the decomposition are regular, indeed they all have compact spectrum and thus are entire functions. On classical functions spaces this regularisation effect can be ``measured" by the following Bernstein inequalities.
\end{remark}

\begin{proposition}[Bernstein's inequalities]\label{paracomposition_Notations and functional analysis_bernstein1}
Suppose that $a\in L^p(\d)$ has its spectrum contained in the ball $\set{\abs{\xi}\leq \lambda}$. 

Then $a\in C^\infty$ and for all $\alpha \in  \n$ and $1\leq p \leq q \leq +\infty$, there is $C_{\alpha,p,q}$ (independent of $\lambda$) such that, 
\[\norm{\partial^{\alpha}_x a}_{L^q} \leq C_{\alpha,p,q} \lambda^{\abs{\alpha}+\frac{1}{p}-\frac{1}{q}}\norm{a}_{L^p}.\]
In particular,
\[\norm{\partial^{\alpha}_x a}_{L^q} \leq C_{\alpha} \lambda^{\abs{\alpha}}\norm{a}_{L^p}, \text{ and for $p=2$, $p=\infty$}\]
\[\norm{a}_{L^\infty}\leq C \lambda^{\frac{1}{2}} \norm{a}_{L^2}.\]
If moreover a has it's spectrum is in $ \set{0<\mu \leq \abs{\xi}\leq \lambda}$ then:
\[
 C_{\alpha,q}^{-1} \mu^{\abs{\alpha}}\norm{a}_{L^q}\leq \norm{\partial^{\alpha}_x a}_{L^q} \leq C_{\alpha,q} \lambda^{\abs{\alpha}}\norm{a}_{L^q}.
\]
\end{proposition}

\begin{proposition}\label{paracomposition_Notations and functional analysis_bernstein2}
For all $\mu >0$, there is a constant $C$ such that for all $\lambda>0$ and for all $\alpha \in W^{\mu,\infty}$ with spectrum contained in $\set{\abs{\xi}\geq \lambda}$. one has the following estimate: 
\[\norm{a}_{L^\infty}\leq C \lambda^{-\mu} \norm{a}_{W^{\mu,\infty}}.\]
\end{proposition}

\begin{definition}[Zygmund spaces on $\d$]\label{paracomposition_Notations and functional analysis_def Zygmund spaces on r}
For $r\in \r$ we define the space: 
\[
C^r_*(\d) \subset \sr'(\d), \text{ by }\
C^r_*(\d)=\set{u\in\sr'(\d),\norm{u}_r=\sup_k 2^{kr}\norm{u_k}_\infty<\infty},\]
 equipped with its canonical topology giving it a Banach space structure.\\
 It's a classical result that for $r\notin \n$, $C^r_*(\d)=W^{r,\infty}(\d)$ the classic H{\"o}lder spaces.
\end{definition}

\begin{proposition} \label{paracomposition_Notations and functional analysis_proposition Zygmund spaces on balls}
Let $\b$ be a ball with center 0. There exists a constant C such that for all $r>0$ and for all $(u_q)_{q\in \n}\in \sr'(\d)$ verifying for all $q$:
\[\supp  \hat{u_q} \subset  2^q \b  \text{ and }  (2^{qr}\norm{u_q}_\infty)_{q\in \n} \text{ is  bounded} \]
\[\text{then}, \ u=\sum_q u_q \in C^r_*(\d) \text{ and } \norm{u}_{r} \leq \frac{C}{1-2^{-r}} \displaystyle{sup_{q \in \n}}2^{qr}\norm{u_q}_\infty. \]
\end{proposition}

\begin{definition}[Sobolev spaces on $\d$]\label{paracomposition_Notations and functional analysis_def Sobolev spaces on r}
It is also a classical result that for $s\in \r$ :
\[H^s(\d)=\set{u\in\sr'(\d),\abs{u}_s= \bigg(\sum_k 2^{2ks} {\norm{u_k}_{L^2}}^2 \bigg)^{\frac{1}{2}}<\infty},\]
 with the right hand side equipped with its canonical topology giving it a Hilbert space structure and $\abs{\ }_s$ is equivalent to the usual norm on $\norm{\ }_{H^s}$.
\end{definition}

\begin{proposition} \label{paracomposition_Notations and functional analysis_proposition Sobolev spaces on balls}
Let $\b$ be a ball with center 0. There exists a constant C such that for all $s>0$ and for all $(u_q)_{\in \n}\in \sr'(\d)$ verifying for all $q$:
\[ \supp \hat{u_q} \subset  2^q \b \text{ and } (2^{qs}\norm{u_q}_{L^2})_{q\in \n} \text{ is in} \ L^2(\n), \]
\[\text{then}, \ u=\sum_q u_q \in H^s(\d) \text{ and } \abs{u}_s \leq \frac{C}{1-2^{-s}} \bigg(\sum_q 2^{2qs} {\norm{u_q}_{L^2}}^2 \bigg)^{\frac{1}{2}}. \]
\end{proposition}

We recall the usual nonlinear estimates in Sobolev spaces:
\begin{itemize}
\item If $u_j\in H^{s_j}(\d), j=1,2$, and $s_1+s_2>0$ then $u_1u_2 \in H^{s_0}(\d)$ and if
\[ s_0\leq s_j, j=1,2 \text{ and } s_0\leq s_1+s_2-\frac{1}{2}, \ \ \  \]
\[\text{then }  \norm{u_1u_2}_{H^{s_0}}\leq K \norm{u_1}_{H^{s_1}}\norm{u_2}_{H^{s_2}} ,\]
where the last inequality is strict if $s_1$ or $s_2$  or $-s_0$ is equal to $\frac{1}{2}$.
\item For all $C^\infty$ function F vanishing at the origin, if $u \in H^s(\d)$ with $s>\frac{1}{2}$ then
\[ \norm{F(u)}_{H^s} \leq C(\norm{u}_{H^s}),\]
for some non decreasing  function C depending only on F.
\end{itemize}

\subsection{Paradifferential operators}\label{paracomposition_Notions of microlocal analysis_Paradifferential Calculus}
We start by the definition of symbols with limited spatial regularity. Let $\w\subset \sr'$ be a Banach space.
\begin{definition}\label{paracomposition_Notions of microlocal analysis_Paradifferential Calculus_ def para symbol}
Given $m \in \r$, $\Gamma^m_\w(\d)$ denotes the space of locally bounded functions $a(x,\xi)$ on $\d\times (\hat{\d}\setminus 0)$, which are $C^\infty$ with respect to $\xi$ for $\xi \neq 0$ and such that, for all $\alpha \in \n$ and for all $\xi \neq 0$, the function $x \mapsto \partial^\alpha_\xi a(x,\xi)$ belongs to $\w$ and there exists a constant $C_\alpha$ such that, for all $\eps>0$:
\begin{equation}\label{paracomposition_Notions of microlocal analysis_Paradifferential Calculus_ definition growth xi condition para} 
\text{ for } \abs{\xi}>\eps, \norm{\partial^\alpha_\xi a(.,\xi)}_{\w}\leq C_{\alpha,\eps} (1+\abs{\xi})^{m-\abs{\alpha}}. 
\end{equation}
The spaces $\Gamma^m_\w(\d)$ are equipped with their natural Fr\'echet topology induced by the semi-norms defined by the best constants in \eqref{paracomposition_Notions of microlocal analysis_Paradifferential Calculus_ definition growth xi condition para} .
We will essentially work with $\w=W^{\rho,\infty}$ and write $\Gamma^m_\w=\Gamma^m_\rho$, for $\rho<0$ we use $\w=C_*^\rho$.
\end{definition}

For quantitative estimates we introduce as in \cite{Metivier08}:
\begin{definition}\label{paracomposition_Notions of microlocal analysis_Paradifferential Calculus_ definition semi-norms}
For $m\in \r$ and $a \in \Gamma^m_\w(\d)$, we set,
\[M^m_\w(a;n)=\sup_{\abs{\alpha}\leq n} \ \sup_{\abs{\xi}\geq\frac{1}{2}}\norm{(1+\abs{\xi})^{m-\abs{\alpha}}\partial^\alpha_\xi a(.,\xi)}_{\w}, \text{ for } n\in \n.\]
We will essentially work with $\w=W^{\rho,\infty},\rho \geq 0$ and write: 
\[
\Gamma^m_{W^{\rho,\infty}}(\d)=\Gamma^m_\rho(\d) \text{ and }
M^m_\rho(a)=M^m_{W^{\rho,\infty}}(a;1).
\]
Moreover we introduce the following spaces equipped with their natural Fr\'echet space structure:
\[
C^{\infty}_b(\d)=\cap_{\rho \geq 0}W^{\rho,\infty}, \ \Gamma^m_\infty(\d)=\cap_{\rho \geq 0}\Gamma^m_\rho(\d), \ \Gamma^{-\infty}_\rho(\d)=\cap_{m\in \r}\Gamma^m_\rho(\d) \text{ and,}
\]
\[
\Gamma^{-\infty}_\infty(\d)=\cap_{\rho \geq 0}\cap_{m\in \r}\Gamma^m_\rho(\d). 
\]
\end{definition}

\begin{remark}
In higher dimension the $1$ in the definition of $M^m_\rho$ should be replaced by $1+\lfloor \frac{d}{2}\rfloor$.
\end{remark}

\begin{definition}
Define an admissible cutoff function as a function $\psi^{B,b}\in C^\infty$,  $B>1,b>0$ that verifies:
\begin{enumerate}
\item 
\[
\psi^{B,b}(\eta,\xi)=0 \text{ when }
\abs{\xi}< B\abs{\eta}+b,
\text{ and }
\psi^{B,b}(\eta,\xi)=1 \text{ when } \abs{\xi}>B\abs{\eta}+b+1.
\]
\item for all $(\alpha,\beta)\in \n^d \times \n^d,$ there is $C_{\alpha_\beta}$, with $C_{0,0}\leq 1$, such that:
\begin{equation}\label{paracomposition_Notions of microlocal analysis_Paradifferential Calculus_definition cutoff growth hypothesis}
\abs{\partial_\xi^\alpha \partial_\eta^\beta \psi(\xi,\eta)}\leq C_{\alpha,\beta} (1+\abs{\xi})^{-\abs{\alpha}-\abs{\beta}}.
\end{equation}
\end{enumerate}
\end{definition}

\begin{definition-proposition}\label{paracomposition_Notions of microlocal analysis_Paradifferential Calculus_def para op}
Consider a real number $m\in \r$, a symbol $a\in \Gamma^m_\w$ and an admissible cutoff function $\psi^{B,b}$ define the paradifferential operator $T_a$ by:
\[\widehat{T_a u}(\xi)=(2\pi)\int_{\hat{\d}}\psi^{B,b}(\xi-\eta,\eta)\hat{a}(\xi-\eta,\eta)\hat{u}(\eta)d\eta,\]
where $\hat{a}(\eta,\xi)=\int e^{-ix.\eta}a(x,\xi)dx$ is the Fourier transform of $a$ with respect to the first variable. 
In the language of pseudodifferential operators:
\[T_a u=op(\sigma_a)u, \text{ where } \fr_x\sigma_a(\xi,\eta)=\psi^{B,b}(\xi,\eta) \fr_x a(\xi,\eta).\]
The connection between two different choices of cut-offs is the following:
\begin{equation}\label{paracomposition_Notions of microlocal analysis_Paradifferential Calculus_difference between 2 choices of cutoff}
\text{ for } a \in \Gamma^m_\rho, (B,B',b,b')\in ]1,+\infty[^2\times ]0,+\infty[^2, \ \sigma^{\psi^{B,b}}_a-\sigma^{\psi^{B',b'}}_a\in \Gamma^{m-\rho}_0.
\end{equation}
\end{definition-proposition}

An important property of paradifferential operators is their action on functions with localised spectrum.
\begin{lemma}\label{paracomposition_Notions of microlocal analysis_Paradifferential Calculus_propostion para action spectrum}
Consider two real numbers $m\in \r$, $\rho\geq 0$, a symbol $a\in \Gamma^m_0(\d)$, an admissible cutoff function $\psi^{B,b}$ and $u \in \sr(\d)$.
\begin{itemize}
\item For $R>>b$, if $\supp \fr u \subset \set{\abs{\xi}\leq R},$ then: 
\begin{equation}\label{paracomposition_Notions of microlocal analysis_Paradifferential Calculus_propostion para action spectrum on rings}
\supp \fr T_a u \subset \set{\abs{\xi}\leq (1+\frac{1}{B})R-\frac{b}{B}},
\end{equation}
\item For $R>>b$, if $\supp \fr u \subset \set{\abs{\xi}\geq R},$ then: 
\begin{equation}\label{paracomposition_Notions of microlocal analysis_Paradifferential Calculus_propostion para action spectrum on balls}
\supp \fr T_a u \subset \set{\abs{\xi}\geq (1-\frac{1}{B})R+\frac{b}{B}},
\end{equation}
\end{itemize}
\end{lemma}

The main features of symbolic calculus for paradifferential operators are given by the following theorems taken from \cite{Metivier08} and \cite{Ayman18}.
\begin{theorem}\label{paracomposition_Notions of microlocal analysis_Paradifferential Calculus_para continuity}
Let $m \in \r$. if $a\in \Gamma^m_0(\d)$, then $T_a$ is of order m. Moreover, for all $\mu \in \r$ there exists a constant K such that:
\[\norm{T_a}_{H^\mu \rightarrow H^{\mu-m}}\leq K M^m_0(a),\text{ and,}\]
\[\norm{T_a}_{W^{\mu,\infty} \rightarrow W^{\mu-m,\infty}}\leq K M^m_0(a), \mu \notin \n.\]
\end{theorem}
\begin{theorem} \label{paracomposition_Notions of microlocal analysis_Paradifferential Calculus_symbolic calculus para precised} 
Let $m,m' \in \r$, and $\rho>0$, $a \in \Gamma^m_\rho(\d)$and $b \in \Gamma^{m'}_\rho(\d)$. 
\begin{itemize}
\item Composition: Then $T_a T_b$ is a paradifferential operator with symbol: $$a \otimes b\in \Gamma^{m+m'}_\rho(\d),\text{ more precisely,}$$
\[
T^{\psi^{B,b}}_a T^{\psi^{B',b}}_b= T^{\psi^{\frac{BB'}{B+B'+1},b}}_{a\otimes b}.
\]
Moreover $T_a T_b- T_{a\#b}$ is of order $m+m'-\rho$ where $a \#b $ is defined by:
\[a \#b=\sum_{\abs{\alpha}<\rho }\frac{1}{i^{\abs{\alpha}}\alpha!} \partial^\alpha_\xi a \partial^\alpha_x b, \]
and there exists $r\in \Gamma^{m+m'-\rho}_0(\d)$ such that:
\[ M^{m+m'-\rho}_0(r) \leq K (M^m_\rho (a) M^{m'}_0(b)+M^m_\rho (a) M^{m'}_0(b)), \]
and we have
 \[T^{\psi^{B,b}}_a T^{\psi^{B',b}}_b- T^{\psi^{\frac{BB'}{B+B'+1},b}}_{a\#b}=T^{\psi^{\frac{BB'}{B+B'+1},b}}_r, \]

\item  Adjoint: The adjoint operator of $T_a$, $T_a^*$ is a paradifferential operator of order m with  symbol $a^*$ defined by:
\begin{equation}\label{paracomposition_Notions of microlocal analysis_Paradifferential Calculus_definition adjoint para}
a^*=\sum_{\abs{\alpha}<\rho} \frac{1}{i^{\abs{\alpha}}\alpha!}\partial^\alpha_\xi \partial^\alpha_x \bar{a}. 
\end{equation}
Moreover, for all $\mu \in \r$ there exists a constant K such that
\[ \norm{T_a^*-T_{a^*}}_{H^\mu \rightarrow H^{\mu-m+\rho}} \leq K M^m_\rho (a). \]
\end{itemize}
\end{theorem}
If $a=a(x)$ is a function of $x$ only, the paradifferential operator $T_a$ is called a paraproduct. 
It follows from Theorem \ref{paracomposition_Notions of microlocal analysis_Paradifferential Calculus_symbolic calculus para precised} and the Sobolev embedding that:
\begin{itemize}
\item If $a \in H^\alpha(\d)$ and $b \in H^\beta(\d)$ with $\alpha,\beta>\frac{1}{2}$, then
\[T_aT_b-T_{ab} \text{ is of order } -\bigg( min\set{\alpha,\beta}-\frac{1}{2} \bigg).\]
\item If $a \in H^\alpha(\d)$ with $\alpha>\frac{1}{2}$, then
\[T_a^*-T_{a^*} \text{ is of order } -\bigg(\alpha-\frac{1}{2} \bigg).\]
\item If $a \in W^{r,\infty}(\d)$, $r\in \n$ then:
\[\norm{au-T_au}_{H^r} \leq C \norm{a}_{W^{r,\infty}} \norm{u}_{L^2}.\]
\end{itemize}
An important feature of paraproducts is that they are well defined for function $a=a(x)$ which are not $L^\infty$ but merely in some Sobolev spaces $H^r$ with $r<\frac{d}{2}$.
\begin{proposition}
Let $m>0$. If $a\in H^{\frac{1}{2}-m}(\d)$ and $u \in H^\mu(\d)$ then $T_au \in  H^{\mu-m}(\d)$. Moreover,
\[ \norm{T_a u}_{H^{\mu -m}}\leq K \norm{a}_{H^{\frac{1}{2} -m}}\norm{u}_{H^{\mu}} \]
\end{proposition}

A main feature of paraproducts is the existence of paralinearisation theorems which allow us to replace nonlinear expressions by paradifferential expressions, at the price of error terms which are smoother than the main terms.

\begin{theorem} \label{paracomposition_Notions of microlocal analysis_Paradifferential Calculus_paralinearisation para product} 
Let $\alpha, \beta \in \r $ be such that $\alpha,\beta> \frac{1}{2}$, then
\begin{itemize}
\item Bony's linearisation theorem: For all $C^\infty$ function F, if $a \in H^\alpha (\d)$ then;
\[ F(a)- F(0)-T_{F'(a)}a \in H^{2\alpha-\frac{1}{2}} (\d). \]
\item If $a\in H^\alpha(\d)$ and $b\in H^\beta(\d)$, then $ab-T_ab-T_ba \in H^{\alpha+ \beta-\frac{1}{2}} (\d)$. Moreover there exists a positive constant K independent of a and b such that:
\[\norm{ab-T_ab-T_ba}_{H^{\alpha+ \beta-\frac{1}{2}} }\leq K  \norm{a}_{H^\alpha} \norm{b}_{H^\beta}  .\]
\end{itemize}
\end{theorem}

\section{Continuity of limited regularity paradifferential exotic symbols on $L^p$ spaces}\label{Cont of lim reg exotic sym}

We start by giving the following analogue of Theorem $2.1.A$ of \cite{Taylor91}.
\begin{theorem}\label{Cont of lim reg exotic sym_thm for gen symb}
Consider four real numbers $r>0,m\in \r$ and $0\leq \delta,\rho \leq 1$, then for all $a(x,\xi)\in C^{r}_* S^m_{\rho,\delta}$ such that $a^*(x,\xi)\in C^{r}_* S^m_{\rho,\delta}$ where:
$$
a^*(x,\xi)=\frac{1}{2\pi}\int_{\d\times \hat{\d}}e^{-iy.\eta} \bar{a}(x-y,\xi-\eta)dyd\eta,
$$
then,
\[\op(a):W^{s+m+(\frac{1}{2}-\frac{1}{p})(1-\rho),p}\rightarrow W^{s,p}, \text{ with }p\in[2,+\infty] \]
provided $0<s<r$. Furthermore, under these hypothesis,
\[
\op(a):C^{s+m+\frac{1}{2}(1-\rho)}_*\rightarrow C^{s}_*.
\]
Moreover there exists a constant K such that:
\[\norm{\op(a)}_{W^{s+m+(\frac{1}{2}-\frac{1}{p})(1-\rho),p}\rightarrow W^{s,p}}\leq K \ {}^* M^{m,r}_{\rho,\delta}(a;1),\text{ and,}\]
\[\norm{\op(a)}_{C^{s+m+\frac{1}{2}(1-\rho)}_*\rightarrow C^{s}_*}\leq K \ {}^* M^{m,r}_{\rho,\delta}(a;1).  \]
\end{theorem}
\begin{remark}
In higher dimension the factor $(\frac{1}{2}-\frac{1}{p})(1-\rho)$ should be adapted to $d(\frac{1}{2}-\frac{1}{p})(1-\rho)$ and the semi norm of order $1$ in the $\xi$ variable in the estimates should be adapted to $\lfloor \frac{d}{2} \rfloor +1$.

If moreover $\delta<1$ then all of the previous results extend to $L^2$ continuity (that is $s=0$), this results from an almost orthogonal decomposition combined with a $T T^*$ argument as shown in Theorem $2$, Section $2.5$ of \cite{Stein93}.

It's a result by H\"ormander \cite{Hormander97} that if $\delta<\rho$ or $\delta=\rho<1$ the hypothesis on $a^*$ is automatically verified. This hypothesis is also shown to be necessary for $\rho=\delta=1$. 
\end{remark}
\begin{proof}
We first notice that it suffices to make $H^s$ and $C^s_*$ estimates as the $L^p$ are obtained directly by interpolation.

 The key estimate follows from the following adaptation of Lemma $4.3.2$ of \cite{Metivier08}:
\begin{lemma}\label{Cont of lim reg exotic sym_key lem Metivier}
There are constants $C$ and $C'$ such that, for all $\lambda>0$ and $q\in C^\infty(\r^d\times \r^d)$ satisfying:
\[
\supp q \subset \r^d \times \set{\abs{\xi}\leq \lambda},\ \ M=\sup_{\abs{\beta}\leq \tilde{d}}\lambda^{\rho \abs{\beta}}\norm{\partial_\xi^\beta q}_{L^\infty}<\infty,
\text{ with }\tilde{d}=\bigg \lfloor \frac{d}{2} \bigg \rfloor+1.\]
Suppose moreover that $q$ and it's derivatives in $\xi$ are uniformly continuous on $\r^d \times \r^d$.
 Then the function,
\[
Q(y)=\int e^{-iy\cdot \xi}q(y,\xi)d\xi,
\]
satisfies:
\begin{equation}\label{Cont of lim reg exotic sym_key lem Metivier_eq1}
\int(1+\abs{\lambda y}^2)^{\tilde{d}}\abs{Q(y)}^2dy\leq C \lambda^d M^2 (1+\lambda^{2(1-\rho)})^{\tilde{d}},
\end{equation}
and,
\begin{equation}\label{Cont of lim reg exotic sym_key lem Metivier_eq2}
\norm{Q}_{L^1(\r^d)}\leq C'M(1+\lambda^{2(1-\rho)})^{\frac{\tilde{d}}{2}}.
\end{equation}
\end{lemma}
\begin{proof}[Proof of Lemma \ref{Cont of lim reg exotic sym_key lem Metivier}]
For $\abs{\alpha}\leq \tilde{d}$ we have:
\[
y^\alpha Q(y)=\int e^{-i y\cdot \xi }D^{\alpha}_\xi q(y,\xi)d\xi.
\]
At this step we would like to apply Plancherel's theorem to deduce:
\begin{equation}\label{Cont of lim reg exotic sym_key lem Metivier_proof_eq1}
\int \abs{y^{2\alpha}}\abs{Q(y)}^2 dy\leq C \lambda^{d-2\rho \alpha}M^2,
\end{equation}
which is the argument given in \cite{Metivier08}, as the application of the Plancherel's theorem does not seem immediate to us we opted to expand upon it to make it's application more immediate. We first notice that it suffice to prove \eqref{Cont of lim reg exotic sym_key lem Metivier_proof_eq1} for $\alpha=0$. To do so we introduce the function:
\[
\tilde{Q}(x,y)=\int e^{-iy\cdot \xi}q(x,\xi)d\xi=\fr_{\xi}q(x,y),
\]
in this setting we can apply Plancherel's theorem to deduce:
\begin{equation}\label{Cont of lim reg exotic sym_key lem Metivier_proof_eq2}
\norm{Q}_{L^\infty_x L^2_y}\leq \norm{q}_{L^\infty_x L^2_\xi}.
\end{equation}
getting back to \eqref{Cont of lim reg exotic sym_key lem Metivier_proof_eq1} we want to estimate $\norm{\tilde{Q}(y,y)}^2_{L_y^2(\r^d)}$, to do so we estimate uniformly on cubes the norms $\norm{\tilde{Q}(y,y)}^2_{L_y^2(C(y_0,R))}$. For this we define the set:
\[
K(y_0,R,\eps)=\set{(x,y), y \in C(y_0,R), \abs{x_j-y_j}\leq \eps, j\in [1,\cdots,d]}.
\]
Thus by the fundamental Theorem of calculus:
\[
\frac{1}{c_d \eps^d}\int_{K(y_0,R,\eps)}\abs{\tilde{Q}(x,y)}^2dx dy \underset{\eps\rightarrow 0}\longrightarrow \int_{C(y_0,R)} \abs{\tilde{Q}(y,y)}^2dy
\]
\begin{align*}
\norm{\tilde{Q}(y,y)}^2_{L_y^2(C(y_0,R))}&\leq C_{K(y_0,R,\eps)} \frac{1}{\eps^{d}} \norm{\tilde{Q}(x,y)}_{L^2(K(R,\eps))}^2\\
&\leq C_{K(R,\eps)} \norm{q}_{L^\infty_x(C(0,\eps); L^2_\xi(C(0,R))}^2,
\end{align*}
where the constant $C_{K(y_0,R,\eps)}$ can be chosen  uniformly in $y_0$ by the uniform continuity of $q$. Now to explicit the dependence of $C_{K(R,\eps)}$ on the different parameters, by Plancherel's theorem we have:
\begin{align*}
\frac{R^d}{c'_d\eps^d}\int_{\r^d}\abs{\prod^d_{j=1}\text{sinc}\bigg(2 \xi_j R\bigg)e^{iy^j_0\xi_j}*q(x,\xi)}^2dx d\xi&=\frac{1}{c_d \eps^d}\int_{K(R,\eps)}\abs{\tilde{Q}(x,y)}^2dx dy\\
& \underset{\eps\rightarrow 0}\longrightarrow \int_{C(0,R)} \abs{\tilde{Q}(y,y)}^2dy.
\end{align*}
Thus we cover the diagonal $(y,y)$ in $\r^d\times \r^d$ by compact sets $(K(y_0^i,R_i,\eps_i))_{i\in \n}$ where $y_0^i, R_i$ and $\eps_i$ are chosen in a manner to ensure that the sum of the volume of the different intersections between elements of this cover is summable.

Thus we have:
\begin{equation*}
\int \abs{y^{2\alpha}}\abs{Q(y)}^2 dy\leq C \lambda^{d-2\rho \alpha}M^2.
\end{equation*}
Multiplying by $\lambda^{2\abs{\alpha}}$ and summing in $\alpha$, implies \eqref{Cont of lim reg exotic sym_key lem Metivier_eq1}. Since $\tilde{d}>\frac{d}{2}$, the second estimate \eqref{Cont of lim reg exotic sym_key lem Metivier_eq2} follows.
\end{proof}

Getting back to the proof of Theorem \ref{Cont of lim reg exotic sym_thm for gen symb}, we are working with $d=\tilde{d}=1$ in Lemma \ref{Cont of lim reg exotic sym_key lem Metivier}.

We start by making a Littlewood-Paley type decomposition by writing:
\[
a(x,\xi)=a(x,\xi)P_0(\xi)+\sum_{k=1}^{\infty}a(x,\xi)P_k(\xi)=a_0(x,\xi)+\sum_{k=1}^{\infty}a_k(x,\xi).
\]

The proof will be divided in two paragraphs where we do the Sobolev and Zygmund estimates respectively.
\paragraph*{\textbf{Continuity in $H^s$}}
\begin{lemma}\label{Cont of lim reg exotic sym_thm for gen symb_proof_cont on Sob_lem}
For $k \geq 0,$ $\op(a_k)$ maps to $L^2$ to $H^\infty$. Moreover for all $\alpha \in \n$, there is $C_\alpha$ such that for all $a \in C^{r}_* S^m_{\rho,\delta}$, $k\geq 0$ and all $f \in L^2$:
\begin{equation}\label{Cont of lim reg exotic sym_thm for gen symb_proof_cont on Sob_lem_eq1}
\norm{\partial^\alpha_x \op(a_k) f}_{L^2}\leq C_\alpha M^{m,\alpha}_{\rho,\delta}(a;1)
\norm{f}_{L^2}2^{k(m+\alpha)}
\end{equation}
\end{lemma}
\begin{proof}[Proof of Lemma \ref{Cont of lim reg exotic sym_thm for gen symb_proof_cont on Sob_lem}]
Since $a_k$ is compactly supported in $\xi$, one sees that $\op(a_k) f$ is given by the convergent integral:
\begin{equation}\label{Cont of lim reg exotic sym_thm for gen symb_proof_cont on Sob_lem_proof_eq1}
\op(a_k) f(x)=\int A_k(x,y)f(y)dy,
\end{equation}
where the kernel $A_k(x,y)$ is given by the convergent integral:
\begin{equation}\label{Cont of lim reg exotic sym_thm for gen symb_proof_cont on Sob_lem_proof_eq2}
\op(a_k)=\frac{1}{2\pi}\int e^{i(x-y) \xi} a_k(x,\xi)d\xi.
\end{equation}

Moreover on the support of $a_k$, $1+\abs{\xi}\simeq 2^k$. Therefore Lemma \ref{Cont of lim reg exotic sym_key lem Metivier} can be applied with $\lambda=2^{k+1}$, implying that:
\begin{equation}\label{Cont of lim reg exotic sym_thm for gen symb_proof_cont on Sob_lem_proof_eq3}
\int(1+2^{2k}\abs{x-y}^2)\abs{A_k(x,y)}^2dy\leq C 2^{2km+k+2(1-\rho)k }  M^{m,0}_{\rho,\delta}(a;1)^2 .
\end{equation}

Hence for $f \in \sr(\d),$ Cauchy-Schwartz inequality implies that:
\begin{equation}\label{Cont of lim reg exotic sym_thm for gen symb_proof_cont on Sob_lem_proof_eq4}
\abs{\op(a_k) f(x)}^2\leq C 2^{2km+k+2(1-\rho)k  }  M^{m,0}_{\rho,\delta}(a;1)^2 \int\frac{ 2^{2km+k+2(1-\rho)k  } \abs{f(y)}^2}{(1+2^{2k}\abs{x-y}^2)}dy .
\end{equation}

The integral $2^{k}\int (1+2^{2k}\abs{x-y}^2)^{-1}dx=C'$ is finite and independent of $k$. Thus:
\begin{equation}\label{Cont of lim reg exotic sym_thm for gen symb_proof_cont on Sob_lem_proof_eq5}
\norm{ \op(a_k) f}_{L^2}\leq C_\alpha M^{m,0}_{\rho,\delta}(a;1)
\norm{f}_{L^2}2^{k(m+(1-\rho))}.
\end{equation}

In order to eliminate the extra factor $2^{k(1-\rho))}$ in \eqref{Cont of lim reg exotic sym_thm for gen symb_proof_cont on Sob_lem_proof_eq5}, we use the fundamental $T T^*$ trick, indeed writing $(a_k)^*$ as the formal symbol of the operator $(\op(a_k))^*$, by the frequency localisation we see that $(a_k)^*\in C^{r}_* S^m_{\rho,\delta}$. Thus $a_k(a_k)^*\in C^{r}_* S^{2m}_{\rho,\delta}$ and applying the previous estimate, as it's uniform in the choice of symbol, to $\op(a_k) (\op(a_k))^*$ we get:
\begin{equation}\label{Cont of lim reg exotic sym_thm for gen symb_proof_cont on Sob_lem_proof_eq6}
\norm{ \op(a_k)(\op(a_k))^* f}_{L^2}\leq C_\alpha M^{m,0}_{\rho,\delta}(a;1)^2
\norm{f}_{L^2}2^{k(2m+(1-\rho))},
\end{equation}
thus by the standard $T T^*$ lemma:
\begin{equation}\label{Cont of lim reg exotic sym_thm for gen symb_proof_cont on Sob_lem_proof_eq7}
\norm{ \op(a_k) f}_{L^2}\leq C_\alpha M^{m,0}_{\rho,\delta}(a;1)
\norm{f}_{L^2}2^{k(m+\frac{1}{2}(1-\rho))}.
\end{equation}

Iterating this estimate we get \eqref{Cont of lim reg exotic sym_thm for gen symb_proof_cont on Sob_lem_eq1} for $\alpha=0$. The symbol $\partial^\alpha_x$ is $(i\xi+\partial_x)^\alpha a_k(x,\xi)$, which gives the desired estimate for larger $\alpha$.
\end{proof}

Now getting back to the continuity in $H^s$ of $\op(a)$ we write for $f\in \sr(\d)$ by the support localisation of $a_k$:
\[
\op(a_k)f=\sum_{\abs{j-k}\leq 3}\op(a_k)P_j f,
\]
Thus by Lemma \ref{Cont of lim reg exotic sym_thm for gen symb_proof_cont on Sob_lem},
\begin{align}\label{Cont of lim reg exotic sym_thm for gen symb_proof_cont on Sob_eq1}
\norm{\partial^\alpha_x \op(a_k) f}_{L^2}&\leq C_\alpha M^{m,\alpha}_{\rho,\delta}(a;1)
\sum_{\abs{j-k}\leq 3}\norm{P_jf}_{L^2}2^{k(m+\alpha)},\nonumber
\intertext{then by Definition \ref{paracomposition_Notations and functional analysis_def Sobolev spaces on r},}
\norm{\partial^\alpha_x \op(a_k) f}_{L^2}&\leq C_\alpha M^{m,\alpha}_{\rho,\delta}(a;1) 2^{k(\alpha-s)}\eps_k,
\end{align}
with,
\begin{equation}\label{Cont of lim reg exotic sym_thm for gen symb_proof_cont on Sob_eq2}
\sum_k \eps_k^2\leq \norm{f}_{H^{s+m}}.
\end{equation}
Now to conclude we recall the following Proposition from \cite{Metivier08}:
\begin{proposition}[Proposition $4.1.13$ of \cite{Metivier08}]\label{Cont of lim reg exotic sym_thm for gen symb_proof_cont on Sob_Proposition 4.1.13}
Let $0<s$ and let $n$ be an integer, $n>s$. There is a constant C such that, for all sequence $(f_k)_{k\geq 0}\in H^n(\d^d)$ satisfying for all $\alpha \in \n^d, \abs{\alpha}\leq n$:
\begin{equation}\label{Cont of lim reg exotic sym_thm for gen symb_proof_cont on Sob_Proposition 4.1.13_eq1}
\norm{\partial_x^\alpha f_k}_{L^2(\d^d)}\leq 2^{k(\abs{\alpha}-s)}\eps_k, \text{ with }(\eps_k)\in l^2,
\end{equation}
the sum $f=\sum f_k$ belongs to $H^s(\d^d)$ and,
\begin{equation}\label{Cont of lim reg exotic sym_thm for gen symb_proof_cont on Sob_Proposition 4.1.13_eq2}
\norm{f}^2_{H^s(\d^d)}\leq C \sum^\infty_{k=0}\eps_k^2.
\end{equation}
\end{proposition}
Applying Proposition \ref{Cont of lim reg exotic sym_thm for gen symb_proof_cont on Sob_Proposition 4.1.13} to $\op(a_k)f$ we get the desired Sobolev continuity and the desired estimate.
\paragraph*{\textbf{Continuity in $C_*^s$}}
The proof follows the same lines as previously, indeed applying \eqref{Cont of lim reg exotic sym_key lem Metivier_eq2} to \eqref{Cont of lim reg exotic sym_thm for gen symb_proof_cont on Sob_lem_proof_eq1} we get the following lemma.
\begin{lemma}\label{Cont of lim reg exotic sym_thm for gen symb_proof_cont on Zyg_lem}
For $k \geq 0,$ $\op(a_k)$ maps to $L^\infty$ to $W^{\infty,\infty}$. Moreover for all $\alpha \in \n$, there is $C_\alpha$ such that for all $a \in C^{r}_* S^m_{\rho,\delta}$, $k\geq 0$ and all $f \in L^\infty$:
\begin{equation}\label{Cont of lim reg exotic sym_thm for gen symb_proof_cont on Zyg_lem_eq1}
\norm{\partial^\alpha_x \op(a_k) f}_{L^\infty}\leq C_\alpha M^{m,\alpha}_{\rho,\delta}(a;1)
\norm{f}_{L^\infty}2^{k(m+\alpha+\frac{1}{2})}.
\end{equation}
\end{lemma}
Again we have:
\begin{align}\label{Cont of lim reg exotic sym_thm for gen symb_proof_cont on Zyg_eq1}
\norm{\partial^\alpha_x \op(a_k) f}_{L^\infty}&\leq C_\alpha M^{m,\alpha}_{\rho,\delta}(a;1) 2^{k(\alpha-s)}\norm{f}_{C_*^{s+m+\frac{1}{2}}},
\end{align}
which gives the desired result and estimate.
\end{proof}

\begin{theorem}\label{Cont of lim reg exotic sym_thm for para symb}
Consider four real numbers $r>0,m\in \r$ and $0\leq \delta,\rho \leq 1$, then for all $a(x,\xi)\in \Gamma^0S^m_{\rho,\delta}$,
\[T_a:W^{s+m+(\frac{1}{2}-\frac{1}{p})(1-\rho),p}\rightarrow W^{s,p}, \text{ with }p\in[2,+\infty], \ \ s\in \r. \]
Furthermore, under these hypothesis,
\[
T_a:C^{s+m+\frac{1}{2}(1-\rho)}_*\rightarrow C^{s}_*,\ \ s\in \r.
\]
Moreover there exists a constant K such that:
\[\norm{T_a}_{W^{s+m+(\frac{1}{2}-\frac{1}{p})(1-\rho),p}\rightarrow W^{s,p}}\leq K \  M^{m,0}_{\rho,\delta}(a;1),\text{ and,}\]
\[\norm{T_a}_{C^{s+m+\frac{1}{2}(1-\rho)}_*\rightarrow C^{s}_*}\leq K \  M^{m,0}_{\rho,\delta}(a;1).  \]
\end{theorem} 

\begin{proof}
This simply follows from the spectral localisation property of paradifferential operators, indeed taking $f\in \sr$, then $\op({\sigma_a}_k)f$ is supported in a ring $C_k$ where $\abs{\xi} \sim 2^k$, which is not necessarily the case for $\op(a_k)f$. The spectral localisation property also ensures that the adjoint operator verifies the hypothesis of Theorem \ref{Cont of lim reg exotic sym_thm for gen symb}. Thus rewriting estimates \eqref{Cont of lim reg exotic sym_thm for gen symb_proof_cont on Sob_eq1} and \eqref{Cont of lim reg exotic sym_thm for gen symb_proof_cont on Zyg_eq1} with $\alpha=0$ then by definition of Sobolev spaces and Zygmund spaces using the Littlewood-Paley decomposition we get: 
\[\norm{T_a}_{W^{s+m+(\frac{1}{2}-\frac{1}{p})(1-\rho),p}\rightarrow W^{s,p}}\leq K \ {}^* M^{m,s}_{\rho,\delta}(\sigma_a;2),\text{ and,}\]
\[\norm{T_a}_{C^{s+m+\frac{1}{2}(1-\rho)}_*\rightarrow C^{s}_*}\leq K \ {}^* M^{m,s}_{\rho,\delta}(\sigma_a;2), \ \ s\in \r,  \]
which gives the desired result by the Bernstein inequalities.
\end{proof}

\end{document}